\documentclass[10pt]{amsart}
\usepackage{graphicx, xcolor}
\usepackage{amssymb}
\numberwithin{equation}{section}

\def\N{\mathbb{N}}
\def\R{\mathbb{R}}

\renewcommand\d{\partial}

\def\k{\kappa}

\def\epsilon{\varepsilon}
\def\e{\varepsilon}

\newcommand\br{\begin{rem}}
\newcommand\er{\end{rem}}
\newcommand\bp{\begin{pmatrix}}
\newcommand\ep{\end{pmatrix}}
\newcommand\be{\begin{equation}}
\newcommand\ee{\end{equation}}
\newcommand\ba{\begin{equation}\begin{aligned}}
\newcommand\ea{\end{aligned}\end{equation}}

\newtheorem{theorem}{Theorem}[section]
\newtheorem{proposition}[theorem]{Proposition}

\newtheorem{remark}[theorem]{Remark}

\title{On the speed rate of convergence of solutions to conservation laws with nonlinear diffusions. } 

\begin{document}

\maketitle

\begin{center}
RAFFAELE FOLINO\footnote{Departamento de Matem\'aticas y Mec\'anica, IIMAS - UNAM, Mexico City, E-mail address: \texttt{folino@mym.iimas.unam.mx}}, 
MARTA STRANI\footnote{Universit\`a Ca' Foscari, Dipartimento di Scienze Molecolari e Nanosistemi, Venezia Mestre (Italy). E-mail address: \texttt{marta.strani@unive.it}}\\

\end{center}
\vskip1cm

\begin{abstract}
In this paper we analyze the long-time behavior of solutions to conservation laws with nonlinear diffusion terms of different types:
saturating dissipation (monotone and non monotone) and singular nonlinear diffusions are considered.
In particular, the cases of mean curvature-type diffusions both in the Euclidean space and in Lorentz-Minkowski space enter in our framework.
After dealing with existence and stability of monotone steady states in a bounded interval of the real line with Dirichlet boundary conditions, 
we discuss the speed rate of convergence to the asymptotic limit as $t\to+\infty$.
Finally, in the particular case of a Burgers flux function, we show that the solutions exhibit the phenomenon of metastability.
\end{abstract}

\begin{quote}\footnotesize\baselineskip 14pt 
{\bf Key words.} 
Saturating diffusion, Minkowski curvature operator, steady states, stability, metastability.
 \vskip.15cm
\end{quote}

\begin{quote}\footnotesize\baselineskip 14pt 
{\bf AMS subject classification.} 
35K20, 35B35, 35B36, 35B40, 35P15.
 \vskip.15cm
\end{quote}

\pagestyle{myheadings}
\thispagestyle{plain}
\markboth{R. FOLINO, M. STRANI}{CONSERVATION LAWS WITH NONLINEAR DIFFUSIONS}

\section{Introduction}
In this paper we are interested in studying the long time dynamical properties of the solutions to a scalar conservation law with a nonlinear diffusion; 
precisely, given $\ell>0$ and $I=(-\ell,\ell)$, we consider the initial boundary value problem
\begin{equation}\label{NonLBurg}
 	\left\{\begin{aligned}
		 u_t+ f(u)_x & =Q(\e u_x)_x,
		&\qquad &x\in I, \,t > 0,\\
 		u(\pm\ell,t)&=u_\pm, &\qquad &t\geq0,\\
		u(x,0)		& =u_0(x), &\qquad &x\in I,
  	\end{aligned}\right.
\end{equation}
where the convection term $f \in C^2(\R)$ and $\e >0$; the main example we have in mind is the Burgers flux $f(u)=u^2/2$.
However, the results of the first part of the paper (Sections \ref{sec2}-\ref{sec3}) hold for a generic function $f \in C^2(\R)$
and it is only in the second part that we focus on a Burgers-type convection. 
Similarly, in the first sections we think of $\e$ as a real parameter, and it is only in Sections \ref{sec4}-\ref{sec5} that we will  consider $\e\ll1$.

As concerning the dissipation flux function $Q$, we will consider three different types of function:
\begin{itemize}
\item  $Q\in C^2(\R)$ monotone increasing and bounded: precisely, we assume
\begin{equation}\label{eq:Qmono}
	Q(0)=0, \qquad \lim_{s\to\pm\infty} Q(s)=\pm Q_\infty, \qquad Q'(s) >0, \quad \forall\,s\in\R, \qquad \lim_{|s|\to\infty}Q'(s)=0.
\end{equation}
The model example (already studied in \cite{FGS}) we have in mind is
\begin{equation}\label{MC}
	Q(s)=\frac{s}{\sqrt{1+s^2}}.
\end{equation}

\item  $Q\in C^2(\R)$ non monotone and bounded; we consider the model example 
\begin{equation}\label{eq:Qnonmonotone}
	Q(s)=\frac{s}{1+s^2}.
\end{equation}

\item $Q\in C^2(-s^*,s^*)$ monotone increasing, unbounded and such that $\displaystyle\lim_{s \to \pm s^*} Q(s)= \pm \infty$ for some $s^*>0$; the model example is 
\begin{equation}\label{eq:Qnonbound}
	Q(s)=\frac{s}{\sqrt{1-s^2}}.
\end{equation}
\end{itemize}
The goal of this paper is twofold: on the one side, we generalize the results of \cite{FGS} to the general setting \eqref{eq:Qmono}; 
moreover, the model we consider here is slightly different with respect to the one proposed in \cite{FGS}, 
so that we are able to overcome the smallness condition on the boundary data stated in \cite{FGS}, as will we see in details later on. 
On the other side, we investigate existence, stability and metastability properties of the steady states also 
in the cases of the dissipation flux functions \eqref{eq:Qnonmonotone} and \eqref{eq:Qnonbound}. 

Before presenting our results, we discuss the three choices of the dissipation flux function $Q$:
first, we focus the attention on the cases \eqref{eq:Qmono}, \eqref{eq:Qnonmonotone}
and then we will comment the case \eqref{eq:Qnonbound}.

The interaction between convection and the saturating dissipation fluxes \eqref{eq:Qmono} and \eqref{eq:Qnonmonotone} 
was originally studied in \cite{KurRos} and \cite{KurLevRos}, respectively. 
If comparing equation
\begin{equation}\label{eq:nonlinearQ}
	u_t+f(u)_x=Q(\e u_x)_x, \qquad \qquad x\in(-\ell,\ell), \;t>0
\end{equation}
with dissipation flux function \eqref{eq:Qmono} or \eqref{eq:Qnonmonotone}, with the classical viscous conservation law 
\begin{equation}\label{eq:visc-cons-law}
	u_t+f(u)_x=\e u_{xx}, \qquad \qquad \e>0,
\end{equation}
(which corresponds to \eqref{eq:nonlinearQ} with the choice $Q(s)=s$), the main novelty 
is that large amplitude solutions develop discontinuities within finite time, while small solutions remain smooth for all times.
Before commenting the previous results obtained for \eqref{eq:nonlinearQ}, 
let us briefly discuss what happens when the convection is absent, i.e. $f=0$ in \eqref{eq:nonlinearQ}.
The equation
\begin{equation}\label{eq:f=0}
	u_t=Q(u_x)_x,
\end{equation}
with $Q$ satisfying \eqref{eq:Qmono} was originally proposed in \cite{Ros}, 
where the author extended the Ginzburg--Landau free-energy functionals to include interaction due to high gradients. 
On the other hand, the case when $Q$ is given by \eqref{eq:Qnonmonotone} was introduced in \cite{PerMal} in the context of image processing.
The solutions of equation \eqref{eq:f=0} can exhibit \emph{hyperbolic phenomena}, such as persistence of discontinuous solutions.
These topics have been studied in \cite{BerDalPas92} when the function $Q$ satisfies \eqref{eq:Qmono}.
Precisely, the authors consider the Cauchy problem for \eqref{eq:f=0} and prove that
if the initial datum $u_0$ is smooth, then the solution remains smooth for all times.
However, if the initial datum is discontinuous, then the solution is continuous for $t>T$, where $T$ could be either $0$ or strictly positive.
The instantaneously smoothness of the solution depends on the degeneracy of $Q$, namely
\begin{align*}
	\int_0^\infty sQ'(s)\,ds=\infty \qquad \Longrightarrow \qquad T=0, \\
	\int_0^\infty sQ'(s)\,ds<\infty \qquad \Longrightarrow \qquad T>0.
\end{align*}
Note that in the model example \eqref{MC} previously studied in \cite{FGS} (the \emph{mean curvature operator in Euclidean space}), we have $T>0$.
The IBVP for the equation \eqref{eq:f=0} and $Q$ given by \eqref{eq:Qnonmonotone} has been studied in \cite[Section 2]{KurLevRos}.

If a nonlinear convection term is added in \eqref{eq:f=0}, that is  equation \eqref{eq:nonlinearQ} is considered, the situation drastically changes.
Indeed, even if we consider a smooth initial datum, the solution may develop discontinuities in a finite time. 
The global existence of a unique smooth solution for the Cauchy problem 
\begin{equation*}
	\begin{cases}
		u_t+f(u)_x=Q(u_x)_x,\\
		u(x,0)=u_0(x),
	\end{cases}
\end{equation*}
where $Q$ either satisfies \eqref{eq:Qmono} or is given by \eqref{eq:Qnonmonotone}, has been investigated in \cite{KurRos} and \cite{KurLevRos}, respectively.
In particular, there exists a unique global (classical) solution if the initial datum is sufficiently small and either compactly supported or periodic.
On the other hand, it is well known that sufficiently large (smooth) initial data generate solutions that become discontinuous in finite time, see \cite{EngSch06,GooKurRos}.
In particular, in \cite{GooKurRos} the authors consider both the Cauchy problem and an IBVP for \eqref{eq:nonlinearQ} 
with either \eqref{eq:Qmono} or \eqref{eq:Qnonmonotone} dissipation fluxes and prove that, 
for certain flux functions $f$ and large initial data $u_0$, there exists a finite breaktime $T > 0$ such that 
\begin{equation*}
	\lim_{t\to T^-}\Vert u_x(\cdot,t)\Vert_{{}_{L^\infty}}=+\infty.
\end{equation*}
Moreover, in \cite{GooKurRos} it was numerically shown that both continuous and discontinuous steady states are strong attractors of a wide class of initial data. 
Finally, let us recall that the blow up in finite time is a consequence of the boundedness of the dissipation flux function $Q$ 
and that if $Q\in C^2(\R)$ is monotone and unbounded, then the derivative $u_x$ remains bounded for all times $t\geq0$.

To conclude this discussion, we briefly mention that equation \eqref{eq:f=0} has been studied also in the presence of a reaction term (see, e.g. \cite{KurRos2}).

Regarding the unbounded and singular case \eqref{eq:Qnonbound}, the nonlinear differential operator  
\begin{equation}\label{mink}
	\textrm{div}\left(\frac{\nabla u}{\sqrt{1-|\nabla u|^2}}\right)
\end{equation}
appears in many applications and it is usually meant as \emph{mean curvature operator in the Lorentz-Minkowski space}.
It is of interest in differential geometry, general relativity and appears in the nonlinear theory of electromagnetism,
where it is usually referred to as Born-Infeld operator; for details see, among others, \cite{BarSim}, \cite{BonDavPom} and references therein.
Recently, the equation
\begin{equation*}
	\textrm{div}\left(\frac{\nabla u}{\sqrt{1-|\nabla u|^2}}\right)+f(u)=0
\end{equation*}
has been extensively studied in many papers and both the boundary value problem (with different boundary conditions) 
and the case of the whole space $\R^n$ with $n\geq1$ are considered.
The bibliography is so rich that it would be impossible to mention all the contributions; here we refer the readers to \cite{Azz, BerJebMaw, BosGar, Huang} and references therein.

On the contrary, the evolution PDE associated to \eqref{mink} (both with or without convection and/or reaction terms) seems, to the best of our knowledge,  
almost unexplored even if it has a considerable appeal from both physical and mathematical point of view.
We are thus here interested in studying the interaction between the {\it mean curvature operator in the Lorentz-Minkowski space} 
and the convection term in the initial boundary value problem \eqref{NonLBurg}, so that to make the literature more complete.
The main difference with respect to the cases \eqref{eq:Qmono} and \eqref{eq:Qnonmonotone} is that 
we have an unbounded and singular dissipation flux function, which satisfies $\min_s Q'(s)=1$. 
As we will see in Section \ref{sec3} (cfr. Proposition \ref{prop:aprioriux-unbound}), 
this implies that smooth initial data do not develop discontinuities in finite time and that the property
$\e\Vert u_x(\cdot,t)\Vert_{{}_{L^\infty}}\leq1$ is preserved for any time $t\geq0$.
As mentioned before, this property is not true in general for the cases \eqref{eq:Qmono} and \eqref{eq:Qnonmonotone},
but an additional condition on the smallness of $\Vert f(u_0)\Vert_{{}_{L^\infty}}$ is needed to have an upper bound on $\Vert u_x(\cdot,t)\Vert_{{}_{L^\infty}}$
(cfr. Propositions \ref{prop:aprioriux-Q} and \ref{prop:aprioriux-nonmono}).

\vskip0.2cm
We close this Introduction with a brief presentation of the main results and the plan of the paper.
Section \ref{sec2} is devoted to the existence of monotone stationary solutions for \eqref{NonLBurg},
and it is divided into three subsections, corresponding to the three different choices of $Q$ \eqref{eq:Qmono}, \eqref{eq:Qnonmonotone}, \eqref{eq:Qnonbound}.
In all the cases, we give necessary and sufficient conditions for the existence of a unique (smooth) monotone stationary solution, 
which is strictly increasing if $u_-<u_+$, and strictly decreasing if $u_->u_+$.
The main difference between the bounded cases \eqref{eq:Qmono}, \eqref{eq:Qnonmonotone} and the unbound case \eqref{eq:Qnonbound}
is that in the first ones we have to impose conditions on the flux functions $f,Q$, on the ratio $\frac{2\ell}\e$ and on the boundary values $u_\pm$, 
while in the last one we only have to ask for $2\ell>\e|u_--u_+|$, (see Propositions \ref{prop:staz-mono}, \ref{prop:ex-nonmono} and \ref{prop:ex-unbound}).
Therefore, in the case \eqref{eq:Qnonbound} we have a smooth connection for any $u_-\neq u_+$, for any flux $f$ and for any $\e>0$, provided that $\ell$ is sufficiently large.
On the contrary, in the cases \eqref{eq:Qmono} or \eqref{eq:Qnonmonotone} even if $\ell$ is large, there exists a smooth steady state 
if and only if $f$ and $u_\pm$ satisfy an appropriate condition (see conditions \eqref{eq:condexmono} and \eqref{eq:condexnonmono}).

In Section \ref{sec3} we study the stability of the steady states introduced in Section \ref{sec2}, and we give sufficient conditions such that the solution 
to the IBVP \eqref{NonLBurg} converges to the monotone steady state as $t\to+\infty$.
As Section \ref{sec2}, Section \ref{sec3} is divided in three subsections, corresponding to the three different choices of the dissipation flux $Q$.
In all the cases, the first step is an \emph{a priori} estimate on $\Vert u_x(\cdot,t)\Vert_{{}_{L^\infty}}$.
Precisely, we have to prove that there is no blow up in finite time for such a quantity.
As it was previously mentioned, in the case \eqref{eq:Qnonbound} the condition $\e\Vert u'_0\Vert_{{}_{L^\infty}}<1$ guarantees global existence of (smooth) solutions,
while in the bounded cases \eqref{eq:Qmono}, \eqref{eq:Qnonmonotone} also a condition on $u_0$ is needed.
Moreover, in order to have stability of the monotone steady states, we also have to impose (in all the three cases) a condition on $f'$ that reads as
\begin{equation}\label{eq:cond-f'}
	\max_{u\in[u_-,u_+]}|f'(u)|\leq C,
\end{equation}
where the constant $C>0$ depends on $Q$.
If considering  \eqref{MC}, \eqref{eq:Qnonmonotone} and \eqref{eq:Qnonbound}, since we have an explicit formula for $Q$,
we can state that the constant $C$ appearing in \eqref{eq:cond-f'} is \emph{proportional} to $\e/\ell^2$. 
This implies that, if $u_\pm$ and $f$ are fixed, we can choose $\ell$ and $\e$ so that
there exists a unique monotone (smooth) steady state $v$ which is \emph{exponentially stable}, namely there exists $K>0$ such that
\begin{equation*}
	\Vert u(\cdot,t)- v\Vert_{{}_{L^2}} \leq e^{-K t} \Vert u_0- v\Vert_{{}_{L^2}}.
\end{equation*}
In conclusion, the condition \eqref{eq:cond-f'} implies \emph{fast} convergence to the asymptotic limit as $t\to+\infty$; 
these results are stated in Theorems \ref{thm:stab-mono}, \ref{thm:stab-nonmono}, \ref{thm:stab-unbound}. 
Let us stress that both in Sections \ref{sec2} and \ref{sec3} we treat the case of a general function $f\in C^2(\R)$ and $\e>0$ does not necessarily need to be small.

Going further, in Section \ref{sec4} we focus the attention on the case when condition \eqref{eq:cond-f'} is not satisfied;  
we consider a Burgers-type convection term and we think of $\e$ as a small parameter.
Motivated by the behavior of the solutions to the viscous Burgers equation with linear diffusion \eqref{eq:visc-cons-law} 
(see, for instance, \cite{FLMS, LafoOMal94, LafoOMal95, MS, ReynWard95, Str14}),
in Section \ref{sec4} we investigate  the phenomenon of \emph{metastability} 
(whereby the time dependent solution reaches its asymptotic configuration in an exponentially, with respect to $\e\ll1$, long time interval) in the case of a nonlinear diffusion. 
In particular, we show that a metastable behavior also appears for the solutions to \eqref{NonLBurg} with dissipation fluxes \eqref{eq:Qnonmonotone}, \eqref{eq:Qnonbound}
(the case of the mean curvature operator in the Euclidean space \eqref{MC} has been recently studied in \cite{FGS});
in these cases condition \eqref{eq:cond-f'} becomes 
\begin{equation}\label{ultima}
	\max_{u\in[u_-,u_+]}|f'(u)|\leq c\e,
\end{equation}
for some $c>0$ independent on $\e$.
This condition is very restrictive since it is satisfied only if $u_\pm$ are small with respect to $\e\ll1$ and, 
if \eqref{ultima} holds true, the solutions do not exhibit a metastable dynamics and they reach the asymptotic limit after a time $T_\e=\mathcal{O}\left(\e^{-1}\right)$;
on the other hand, if $u_\pm$ are sufficiently large, then condition \eqref{ultima} is not satisfied and the stability results of Section \ref{sec3} do not hold anymore. 
In this case we will prove that we still have convergence to the steady state, 
but the time needed for the solutions to reach the asymptotic limit becomes $T_\e=\mathcal{O}\left(\exp(|f'(u_\pm)|/\e)\right)$.
In this case we thus have a {\it slow} convergence to the asymptotic limit as $t \to +\infty$. 

As we will see, this implies that when $T_\e=\mathcal{O}\left(\e^{-1}\right)$ it is possible to accelerate time by a factor $\e^{-1}$, 
and in the new scale time we have convergence after a time which is independent on $\e$; conversely, 
any such kind of acceleration is hopeless in the case $T_\e=\mathcal{O}\left(\exp(|f'(u_\pm)|/\e)\right)$ (for further details see Section \ref{sec:4concl}).

Finally, in Section \ref{sec5} we show some numerical simulations for the solutions to \eqref{NonLBurg} (both in the case \eqref{eq:Qnonmonotone} and \eqref{eq:Qnonbound}), confirming the results obtained in Section \ref{sec4} on the speed rate of the convergence.

\section{Existence of monotone stationary solutions}\label{sec2}
In this section we study the existence of smooth stationary solutions for the problem \eqref{NonLBurg} 
in the three different cases of dissipation flux function $Q$ we presented in the Introduction. 

\subsection{Stationary solutions on the whole line} 
Before stating the results for the IBVP \eqref{NonLBurg}, we consider the equation in the whole real line,
\begin{equation}\label{eq:allR}
	u_t+{f(u)}_x={Q(\e u_x)}_x, \qquad \qquad x\in\R,\; t>0,
\end{equation} 
and we look for traveling wave solutions connecting the values $u_+\neq u_-$; precisely,  we look for solutions of the form $\phi(x-\lambda t)$, 
where the profile $\phi:\R\to\R$ is a monotone function satisfying
\begin{equation*}
	\lim_{x\to\pm\infty}\phi(x)=u_\pm,
\end{equation*}
and the velocity $\lambda\in\R$.
With the change of variable $\xi=x-\lambda t$, it follows that
$$
	-\lambda\phi'+{f(\phi)}_\xi={Q(\e\phi')}_\xi.
$$
Assuming that $\phi'(\pm\infty)=0$ and integrating over $\R$, we deduce that $\lambda$ satisfies the Rankine--Hugoniot conditions
\begin{equation*}
	\lambda=\frac{f(u_+)-f(u_-)}{u_+-u_-}.
\end{equation*}
In particular, we have stationary solutions if and only if $\lambda=0$, namely
\begin{equation*}
	f(u_+)=f(u_-).
\end{equation*}
In this case, the profile $\phi$ satisfies 
\begin{equation}\label{eq:profile-fi}
	Q(\e\phi'(x))=f(\phi(x))+Q(0)-f(u_-), \qquad \qquad x\in\R,
\end{equation}
and the existence of a (smooth) monotone profile $\phi$ depends on the choice of the boundary values $u_\pm$, the flux functions $f$ and the dissipation $Q$.
However, if the function $\phi=\phi(x)$ is a solution to \eqref{eq:profile-fi}, then the functions $\phi_c:=\phi_c(x)=\phi(x-c)$ are solutions to \eqref{eq:profile-fi} as well, for all $c\in\R$.
Regarding the existence, observe that, if considering a convex flux function $f$ and  a dissipation flux function satisfying $Q(s)s>0$,
there are no monotone increasing solutions to \eqref{eq:profile-fi}.

In conclusion, we can say that a necessary condition for the existence of stationary solutions to \eqref{eq:allR} connecting $u_-$ and $u_+$ is $f(u_+)=f(u_-)$;
in particular, if the flux function $f$ is convex, there exist only stationary solutions connecting $u_->u_+$. 
As we will see in the next subsections, the situation is different when considering the problem on a bounded interval.

\subsection{The case of a monotone and bounded dissipation flux function}\label{sec:exQmono}
To start with, we study the existence of smooth stationary solutions to  \eqref{NonLBurg} 
with $Q$ satisfying \eqref{eq:Qmono}. 
We stress that the special case of a mean curvature type diffusion in the Euclidean space \eqref{MC}
has been already studied in \cite{FGS}.

For simplicity, we look for {\it increasing} steady states (the case of decreasing steady states being completely symmetric), and we assume $u_-<u_+$.
We prove the following result.
\begin{proposition}\label{prop:staz-mono}
Let $\e>0$ and $f,Q\in C^2(\R)$, with $Q$ satisfying \eqref{eq:Qmono}, and consider the boundary value problem
\begin{equation}\label{eq:staz-mono-BV}
	\begin{cases}
	{Q(\e v_x)}_x=f(v)_x, \qquad \qquad x\in(-\ell,\ell),\\
	v(\pm\ell)=u_\pm,
	\end{cases}
\end{equation}
for some $u_-<u_+$.
There exists a unique (smooth) increasing solution to \eqref{eq:staz-mono-BV} if and only if 
\begin{equation}\label{eq:condexmono}
	M-m < Q_\infty  \qquad \mbox{ and } \qquad \frac{2\ell}\e>A,
\end{equation}
where 
$$m:=\min_{u \in [u_-, u_+]} f(u), \qquad M:=\max_{u \in [u_-, u_+]} f(u), \qquad A:=\int_{u_-}^{u_+}\frac{ds}{Q^{-1}\left(f(s)+Q_\infty-M\right)}.$$
\end{proposition}

\begin{proof}
Solutions to \eqref{eq:staz-mono-BV} solve
\begin{equation}\label{staz-mono}
	Q(\e v_x)=f(v)+c, \qquad \; v(-\ell)=u_-, \quad v(\ell)=u_+,
\end{equation}
where $c \in \R$ is an integration constant that is uniquely determined by the boundary conditions $u_\pm$. 
Because of the assumptions \eqref{eq:Qmono} on the function $Q$,  one has to impose 
\begin{equation*}
	0< f(v)+ c< Q_\infty, \qquad  \forall\, v \in [u_-,u_+].
\end{equation*}
The latter condition is equivalent to
$$
-m<c< Q_\infty-M,
$$
which implies $M-m < Q_\infty$, that is the first assumption in \eqref{eq:condexmono}.
Going further, since $Q$ is monotone, from \eqref{staz-mono} it follows that
$$\e v_x=Q^{-1}\left(f(v)+c\right),$$
and $v$ is implicitly defined by
$$\int_{u_-}^{v(x)}\frac{ds}{Q^{-1}\left(f(s)+c\right)}=\frac{x+\ell}\e.$$
By imposing $v(\ell)=u_+$ we obtain the following condition
\begin{equation}\label{eq:Fi(c)}
	\int_{u_-}^{u_+}\frac{ds}{Q^{-1}\left(f(s)+c\right)}=\frac{2\ell}\e.
\end{equation}
Denoting by $\Phi(c)$ the function on the left hand side of \eqref{eq:Fi(c)}, there exists a unique solution of \eqref{staz-mono}
if and only if there exists $c_*\in\R$ such that $\Phi(c_*)=2\ell/\e$.
We now observe that the function $\Phi$, defined in $(-m,Q_\infty-M)$, is a continuous function and its derivative is given by
\begin{equation*}
	\Phi'(c)=-\int_{u_-}^{u_+}\frac{\left(Q^{-1}\right)'\left(f(s)+c\right)}{\left[Q^{-1}\left(f(s)+c\right)\right]^2}\,ds<0.
\end{equation*}
Hence $\Phi$ is a decreasing function; moreover
$$\lim_{c\to-m}\Phi(c)=\int_{u_-}^{u_+}\frac{ds}{Q^{-1}\left(f(s)-m\right)}=:A_+,$$
and
$$\lim_{c\to Q_\infty-M}\Phi(c)=\int_{u_-}^{u_+}\frac{ds}{Q^{-1}\left(f(s)+Q_\infty-M\right)}=:A_-.$$
Therefore, $\Phi:(-m,Q_\infty-M)\rightarrow(A_-,A_+)$ and there exists a unique (smooth) solution of \eqref{staz-mono} if and only if
\begin{equation*}
	M-m < Q_\infty, \qquad \qquad A_-<\frac{2\ell}\e<A_+.
\end{equation*}
Observe that, from the definitions of $m$ and $M$, it follows that $f(s)+Q_\infty-M>0$ for any $s\in[u_-,u_+]$, implying that $A_-<\infty$.
On the other hand, there exists $\tilde{s}\in[u_-,u_+]$ such that $f(\tilde{s})-m=0$;
since $f,Q\in C^2(\R)$ with $Q$ satisfying \eqref{eq:Qmono}, we have that $Q^{-1}$ is differentiable in $0$ and, as a consequence, $A_+=+\infty$.
This completes the proof.
\end{proof}
Proposition \ref{prop:staz-mono} gives necessary and sufficient conditions for the existence of a unique smooth increasing stationary solution to \eqref{NonLBurg}.
Similarly, in the case $u_->u_+$ we can find a condition like \eqref{eq:condexmono} for the existence of a unique smooth decreasing stationary solution.
Finally, we stress that if $M-m$ is greater than $Q_\infty$, then it is not possible to connect the values $u_\pm$ with a smooth profile.
 
\begin{remark}\label{rem:ex-staz-mono}
{\rm We make some comments on the assumptions of Proposition \ref{prop:staz-mono}.
First of all, notice that the request $Q'(s)>0$ for all $s$ is fundamental; 
for instance, if $Q(s)\sim s^3$ for $s$ close to $0$ and $f(s)=s^2$, then $A_+<\infty$ and we would have a much more restrictive assumption on $\ell$.

Next, as concerning the first condition in \eqref{eq:condexmono}, we stress that it can be seen either as a restriction on the boundary data $u_\pm$ 
we want to connect (if the functions $Q$ and $f$ are fixed), or as a condition on $Q_\infty$ (that must be taken sufficiently large) if $f$ and $u_\pm$ are fixed. 
We notice that such condition is exactly the one stated in \cite[Theorem 2.1]{FGS}; 
however, since the model considered is slightly different (we here consider $Q(\e u_x)$ instead of  $\e Q(u_x)$, 
with $Q$ given by \eqref{MC}) it does not imply any smallness (with respect to $\e$) assumption on the boundary data $u_\pm$, 
conversely to what happens in \cite{FGS}, where the only smooth solutions connect values $u_\pm$ which are small with respect to $\e$.

Finally, the second condition in \eqref{eq:condexmono} implies that the length of the interval $I$ must be sufficiently large.
We notice that, since $Q^{-1}$ is an increasing function, we have
\begin{equation}\label{eq:A-est}
	A\leq\frac{u_+-u_-}{Q^{-1}(Q_\infty-M+m)}.
\end{equation}
Hence, similarly as before, condition \eqref{eq:condexmono}  is satisfied for any $\e$ either if $u_\pm$ are taken sufficiently close, 
or (in the case $f$ and $u_\pm$ are fixed) if $Q_\infty$ is chosen sufficiently large.
}
\end{remark}

\subsection{The case of a non monotone dissipation flux function}\label{sec:ex-nonmono}
In this subsection we study the existence of smooth stationary solutions to \eqref{NonLBurg} when the function $Q$ is non monotone and bounded.
For simplicity, we fix 
$$
Q(s)=\frac{s}{1+s^2}.
$$
As in Section \ref{sec:exQmono}, we look for {\it strictly monotone} steady states; 
to be more complete, we now focus the attention on the case of decreasing stationary solution 
(as opposite of what we did in Section \ref{sec:exQmono}) and then we assume $u_->u_+$. 
We stress once again that, also in this case, the computations needed to prove the existence of increasing steady states are completely symmetric.

\begin{proposition}\label{prop:ex-nonmono}
Set $\e>0$ and $f\in C^2(\R)$ and consider the boundary value problem
\begin{equation}\label{eq:staz-nonmono-BV}
	\begin{cases}
	\displaystyle\left(\frac{\e v_x}{1+\e^2v_x^2}\right)_x=f(v)_x, \qquad \qquad x\in(-\ell,\ell),\\
	v(\pm\ell)=u_\pm,
	\end{cases}
\end{equation}
for some $u_->u_+$.
There exists a unique (smooth) decreasing solution to \eqref{eq:staz-nonmono-BV} satisfying 
$$-\e^{-1}<v_x(x)<0, \qquad \qquad  \forall\,x\in(-\ell,\ell),$$
if and only if 
\begin{equation}\label{eq:condexnonmono}
	M-m < \frac12 \qquad \mbox{ and } \qquad \frac{2\ell}\e>B,
\end{equation}
where $m=\displaystyle\min_{u \in [u_+, u_-]} f(u)$, $\, M=\displaystyle\max_{u \in [u_+, u_-]} f(u)$ and
$$B:=\int_{u_+}^{u_-}\frac{1+2\sqrt{(f(s)-m)\left(1-(f(s)-m)\right)}}{1-2(f(s)-m)}\,ds.$$
Moreover, there exists a unique (smooth) decreasing solution to \eqref{eq:staz-nonmono-BV} satisfying 
$$v_x(x)<-\e^{-1}, \qquad \qquad  \forall\,x\in(-\ell,\ell),$$
if and only if 
\begin{equation}\label{eq:cond-L}
	M-m < \frac12, \qquad \mbox{ and } \qquad L_-<\frac{2\ell}\e<L_+,
\end{equation}
where
$$
L_-:=\int_{u_+}^{u_-}\!\!\!\frac{2(M-f(s))}{1+\sqrt{1-4\left(f(s)-M\right)^2}}\,ds, \qquad 
L_+:=\int_{u_+}^{u_-}\!\!\!\frac{1-2(f(s)-m)}{1+2\sqrt{(f(s)-m)\left(1-(f(s)-m)\right)}}\,ds.$$
\end{proposition}

\begin{proof}
We proceed as in the proof of Proposition \ref{prop:staz-mono}.
In this case, solutions to \eqref{eq:staz-nonmono-BV} solve
\begin{equation}\label{staz-nonmono}
	\frac{\e v_x}{1+\e^2v_x^2}= f(v)+c,
\end{equation}
where $c \in \R$ is an integration constant; since the left hand side in \eqref{staz-nonmono} is strictly negative and greater than $-1/2$ 
(that is the global minimum of the function $Q(s)=\frac{s}{1+s^2}$), we must assume
\begin{equation}\label{eq:boundcondnonmono}
	-\frac12< f(v)+ c<0, \qquad  \forall\, v \in [u_+,u_-].
\end{equation}
As a consequence, we have  $-m-\frac12<c<-M$, which implies $M-m < 1/2$, that is the first assumption in \eqref{eq:condexnonmono}.
Under this condition, from \eqref{staz-nonmono} it follows
\begin{equation*}
	\e v_x = \frac{1\pm\sqrt{1-4\left(f(v)+c\right)^2}}{2\left(f(v)+c\right)}.
\end{equation*}
Observe that \eqref{eq:boundcondnonmono} implies that $v_x$ is well defined and we have two possibilities:  $\e v_x<-1$ and $-1<\e v_x<0$.
First, let us consider the case $-1<\e v_x<0$; since  $v_x$ satisfies
\begin{equation*}
	\e v_x=\frac{2(f(v)+c)}{1+\sqrt{1-4\left(f(v)+c\right)^2}},
\end{equation*}
then $v$ is implicitly defined by
$$-\int_{v(x)}^{u_-}\frac{1+\sqrt{1-4 \left(f(s)+c\right)^2}}{2(f(s)+c)}\,ds=\frac{x+\ell}\e.$$
Imposing $v(\ell)=u_+$ we obtain the condition
$$\Phi(c):=-\int_{u_+}^{u_-}\frac{1+\sqrt{1-4 \left( f(s)+c\right)^2}}{2(f(s)+c)}\,ds=\frac{2\ell}\e.$$
In this case, the function $\Phi$ is defined in $(-m-\frac12,-M)$ and its derivative,  given by
\begin{equation*}
	\Phi'(c)=\int_{u_+}^{u_-}\frac{1+\sqrt{1-4\left(f(s)+c\right)^2}}{2(f(s)+ c)^2\sqrt{1-4\left(f(s)+c\right)^2}}\,ds,
\end{equation*}
is positive, implying  that $\Phi$ is an increasing function. 
Moreover
$$\lim_{c\to-m-\frac12}\Phi(c)=\int_{u_+}^{u_-}\frac{1+2\sqrt{(f(s)-m)\left(1-(f(s)-m)\right)}}{1-2(f(s)-m)}\,ds=:B,$$
and
$$\lim_{c\to-M}\Phi(c)=\int_{u_+}^{u_-}\frac{1+\sqrt{1-4\left(f(s)-M\right)^2}}{2(M-f(s))}\,ds=:B_+.$$
Since $f\in C^2(\R)$, the latter integral is infinite, namely $B_+=+\infty$.
Therefore, there exists a unique decreasing (smooth) solution $v$ to \eqref{staz-nonmono} satisfying $v_x>-\e^{-1}$ if and only if condition \eqref{eq:condexnonmono} holds.

We now consider the case $v_x<-\e^{-1}$; similarly as before one obtains 
\begin{equation*}
	\e v_x = \frac{1+\sqrt{1-4\left(f(v)+c\right)^2}}{2\left(f(v)+c\right)},
\end{equation*}
together with the condition
\begin{equation*}
	\Psi(c):=-\int_{u_+}^{u_-}\frac{2(f(s)+c)}{1+\sqrt{1-4 \left(f(s)+c\right)^2}}\,ds=\frac{2\ell}\e.
\end{equation*}
In this case the function $\Psi$ is decreasing, implying that there exists a unique decreasing (smooth) solution $v$ to \eqref{staz-nonmono} satisfying $v_x<-\e^{-1}$ if and only if
condition \eqref{eq:cond-L} holds, where $L_{\pm}$ are defined by
$$L_-:=\lim_{c\to-M}\Psi(c)=\int_{u_+}^{u_-}\frac{2(M-f(s))}{1+\sqrt{1-4 \left(f(s)-M\right)^2}}\,ds,$$
and
$$L_+:=\lim_{c\to-m-\frac12}\Psi(c)=\int_{u_+}^{u_-}\frac{1-2(f(s)-m)}{1+2\sqrt{(f(s)-m)\left(1-(f(s)-m)\right)}}\,ds.$$
The proof is now complete.
\end{proof}
\begin{remark}\label{rem:ex-staz-nonmono}
{\rm
As concerning the second condition in assumptions \eqref{eq:condexnonmono}, we have 
\begin{equation}\label{eq:B-est}
	u_--u_+\leq B\leq\frac{1+2\sqrt{(M-m)(1-(M-m))}}{1-2(M-m)}(u_--u_+),
\end{equation}
Hence, in order for the second assumption in \eqref{eq:condexnonmono} to be satisfied for any choice of $\e$ and $\ell$ 
(that is, in order to have $B\ll1$), one has to take $u_\pm$ close.
On the other hand, if $f$ and $u_\pm$ are fixed, we have to choose  $\ell/\e$ large in order to have \eqref{eq:condexnonmono}.

Regarding assumptions \eqref{eq:cond-L}, we have the following estimates for the constants $L_-, L_+$:
\begin{equation*}
 	L_-\leq\frac{2(M-m)}{1+\sqrt{1-4 \left(m-M\right)^2}}(u_--u_+), 
\end{equation*}
and 
\begin{equation*}
	\frac{1-2(M-m)}{1+2\sqrt{(M-m)(1-(M-m))}}(u_--u_+)\leq L_+\leq u_--u_+.
\end{equation*}
In particular, 
\begin{equation*}
	\lim_{u_+\to u_-}L_-=\lim_{u_+\to u_-}L_+=0.
\end{equation*}
Roughly speaking, the second assumption in \eqref{eq:cond-L} imposes a very restrictive choice for the length of the interval $\ell$ 
even if we choose either  $u_--u_+\ll1$ or  $\e\ll1$.
}
\end{remark}

\subsection{The case of a monotone and unbounded dissipation flux function}\label{sec:ex-unbound}
In this last subsection we study the existence of smooth stationary solutions to \eqref{NonLBurg} 
when the function $Q$ is monotone and unbounded, and we consider the specific case
$$
Q(s)=\frac{s}{\sqrt{1-s^2}}.
$$
We look again for {\it strictly monotone} steady states, and we here consider increasing stationary solutions ($u_+>u_-$) being, as before, the decreasing case identical.
\begin{proposition}\label{prop:ex-unbound}
Set $\e>0$ and $f\in C^2(\R)$ and consider the boundary value problem
\begin{equation}\label{eq:staz-unbound-BV}
	\begin{cases}
	\displaystyle\left(\frac{\e v_x}{\sqrt{1-\e^2v_x^2}}\right)_x=f(v)_x, \qquad \qquad x\in(-\ell,\ell),\\
	v(\pm\ell)=u_\pm,
	\end{cases}
\end{equation}
for some $u_+>u_-$.
There exists a unique (smooth) increasing solution to \eqref{eq:staz-unbound-BV} if and only if $$\frac{2\ell}\e>u_+-u_-.$$
\end{proposition}

\begin{proof}
Solutions to \eqref{eq:staz-unbound-BV} satisfy
\begin{equation}\label{staz-unbound}
	\frac{\e v_x}{\sqrt{1-\e^2v_x^2}}= f(v)+c,
\end{equation}
where $c \in \R$ is an integration constant, so that we assume 
\begin{equation*}
	f(v)+c>0, \qquad  \forall\, v \in [u_-,u_+],
\end{equation*}
since the left hand side in \eqref{staz-unbound} is strictly positive.
Therefore, if $c>-m$, from \eqref{staz-unbound} it follows 
\begin{equation*}
	\e v_x = \frac{f(v)+c}{\sqrt{1+(f(v)+c)^2}}.
\end{equation*}
Hence, $v$ is implicitly defined by
$$\int_{v(x)}^{u_+}\frac{\sqrt{1+\left(f(s)+c\right)^2}}{f(s)+c}\,ds=\frac{x+\ell}\e,$$
and imposing $v(-\ell)=u_-$ we obtain the condition
$$\Phi(c):=\int_{u_-}^{u_+}\frac{\sqrt{1+\left(f(s)+c\right)^2}}{f(s)+c}\,ds=\frac{2\ell}\e.$$
In this case, the function $\Phi$ is defined in $(-m,+\infty)$ and it is a decreasing function. 
Moreover
$$\lim_{c\to-m}\Phi(c)=\int_{u_-}^{u_+}\frac{\sqrt{1+\left(f(s)-m\right)^2}}{f(s)-m}\,ds=+\infty,$$
and
$$\lim_{c\to+\infty}\Phi(c)=u_+-u_-.$$
Hence, there exists a unique $c_*>-m$ such that $\Phi(c_*)=2\ell$ if and only if $\displaystyle\frac{2\ell}\e>u_+-u_-$.
\end{proof}

As it was mentioned in the Introduction, we stress that in this last case the assumptions needed in order to have existence of a smooth steady state are much less restrictive: 
indeed, for any flux function $f$ and boundary data $u_\pm$, we have a smooth connection if the ratio $\ell/\e$ is large enough.

\section{Stability of monotone stationary solutions}\label{sec3}
This section is devoted to the study of the stability fo the steady states, whose existence has been proved in Section \ref{sec2}; 
as before, we divide the analysis considering separately the three different dissipation fluxes $Q$.
\subsection{Monotone and bounded dissipation flux function}
Our first result proves the stability of the steady states to \eqref{NonLBurg}, with $Q$ satisfying \eqref{eq:Qmono}; 
we refer the reader to \cite{FGS} for the analysis of the special case \eqref{MC}.
We start by proving the following result.
\begin{proposition}\label{prop:aprioriux-Q}
Fix $T > 0$ and let $u(\cdot,t)\in C^3(I)$ be a classical solution to the IBVP \eqref{NonLBurg} for $t\in[0,T]$ with $Q$ satisfying \eqref{eq:Qmono}.
If
\begin{equation}\label{eq:smallness-Q}
	\left\|Q(\e u'_0)\right\|_{{}_{L^\infty}}+2\Vert f(u_0)\Vert_{{}_{L^\infty}}\leq\beta<Q_\infty,
\end{equation}
then 
\begin{equation}\label{stimaux}
	\e\Vert u_x(\cdot,t)\Vert_{{}_{L^\infty}}\leq Q^{-1}(\beta),
\end{equation}
for any $t\in[0,T]$.
\end{proposition}
\begin{remark}\label{rem:ipoex-stab-mono}
{\rm
Let us compare assumption \eqref{eq:smallness-Q} with condition \eqref{eq:condexmono}, which we had to impose in order to obtain the existence of regular steady states.
Assume, without loss of generality that $u_-<u_+$.
First of all, notice that
$$M-m\leq2\Vert f(u_0)\Vert_{{}_{L^\infty}}, \quad \mbox{with} \quad  M=\max_{u\in[u_-,u_+]} f(u) \quad \mbox{and} \quad m=\min_{u\in[u_-,u_+]}f(u).$$
Hence, assumption \eqref{eq:smallness-Q} implies the first condition in \eqref{eq:condexmono}.
Moreover, from \eqref{eq:smallness-Q} it follows
\begin{equation*}
	\left\|Q(\e u'_0)\right\|_{{}_{L^\infty}}\leq\beta-2\Vert f(u_0)\Vert_{{}_{L^\infty}}\leq\beta-(M-m),
\end{equation*}
and, as a consequence
\begin{equation*}
	\e \left\|u'_0\right\|_{{}_{L^\infty}}\leq Q^{-1}\left(\beta-M+m\right).
\end{equation*}
Hence,
\begin{equation*}
	\frac{\e(u_+-u_-)}{2\ell}\leq Q^{-1}\left(\beta-(M-m)\right)<Q^{-1}\left(Q_\infty-M+m)\right),
\end{equation*}
so that, taking into account \eqref{eq:A-est}, we infer
\begin{equation*}
	\frac{2\ell}\e>\frac{u_+-u_+}{Q^{-1}\left(Q_\infty-M+m\right)}\geq A.
\end{equation*}
In conclusion, assumption \eqref{eq:smallness-Q} implies condition \eqref{eq:condexmono}. 
}
\end{remark}

\begin{proof}[Proof of Proposition \ref{prop:aprioriux-Q}]
Similarly as in \cite{FGS}, we define the function
\begin{equation}\label{eq:z-Q}
	z(x,t):=Q\left(\e u_x(x,t)\right)-f(u(x,t)),
\end{equation}
where $u$ is the classical solution to \eqref{NonLBurg} and $Q$ satisfies \eqref{eq:Qmono}.
Therefore, we can rewrite the equation for $u$ in the form
\begin{equation*}
 	u_t=z_x.
\end{equation*}
Differentiating equation \eqref{eq:z-Q} with respect to $t$, we deduce
\begin{equation}\label{eq:z_t-Q}
	z_t=\e Q'(\e u_x)z_{xx}-f'(u)z_x.
\end{equation}
By contradiction, let us assume that there exists $t^*\in(0,T)$ such that $\e\Vert u_x(\cdot,t)\Vert_{{}_{L^\infty}}<C_1$ for $t\in[0,t^*)$ 
and $\e\Vert  u_x(\cdot,t^*)\Vert_{{}_{L^\infty}}=C_1$, for a constant $C_1>Q^{-1}(\beta)$.
As a consequence, the equation in \eqref{NonLBurg} and equation \eqref{eq:z_t-Q} remain parabolic for $t\in[0,t^*]$, and the maximum principle implies
\begin{equation}\label{aprioriu-Q} 
	\Vert u(\cdot,t)\Vert_{{}_{L^\infty}}\leq\Vert u_0\Vert_{{}_{L^\infty}}, \qquad \qquad \forall\,t\in[0,t^*].
\end{equation}
Moreover, the term $f'(u)$ is uniformly bounded and
\begin{equation*}
	\Vert z(\cdot,t)\Vert_{{}_{L^\infty}}\leq\Vert z(\cdot,0)\Vert_{{}_{L^\infty}}=\left\Vert Q(\e u'_0)- f(u_0)\right\Vert_{{}_{L^\infty}}, \qquad \qquad \forall\,t\in[0,t^*].
\end{equation*}
Hence, by using the assumptions \eqref{eq:smallness-Q} and \eqref{aprioriu-Q}, we obtain
\begin{equation*}
	\left\Vert Q\left(\e u_x(\cdot,t)\right)\right\Vert_{{}_{L^\infty}}\leq\beta<Q_\infty, \qquad \qquad \forall\,t\in[0,t^*],
\end{equation*}
and, because of \eqref{eq:Qmono}, we end up with
\begin{equation*}
	\e\Vert u_x(\cdot,t)\Vert_{{}_{L^\infty}}\leq Q^{-1}(\beta), \qquad \qquad \forall\,t\in[0,t^*].
\end{equation*}
This leads to a contradiction, and the proof is complete.
\end{proof}

Denote by $v=v(x)$ the increasing stationary solution of \eqref{NonLBurg} with $Q$ satisfying \eqref{eq:Qmono} defined in Proposition \ref{prop:staz-mono}, namely $v$ solves
\begin{equation*}
	Q(\e v_x)_x-f(v)_x=0, \qquad v(-\ell)=u_-, \quad v(\ell)=u_+.
\end{equation*}
The next goal is to prove that the steady state $v$ is asymptotically stable, that is the solution to \eqref{NonLBurg} converges to $v$ in $L^2$ as $t\to+\infty$.
\begin{theorem}\label{thm:stab-mono}
Let $u$ be a classical solution to the initial boundary value problem \eqref{NonLBurg}, with dissipative flux function $Q$ satisfying \eqref{eq:Qmono},
$u_-<u_+$ and initial datum $u_0 \in C^3(I)$ satisfying \eqref{eq:smallness-Q}. 
Then, there exists a positive constant $K_1$, (depending on $Q, u_0$ and that can be explicitly computed) 
such that, if
\begin{equation}\label{eq:ipo-stab-mono}
	\max_{u \in [u_-, u_+]} |f'(u)| \leq K_1,
\end{equation}
then
\begin{equation*}
	\Vert u(\cdot,t)- v\Vert_{{}_{L^2}} \leq e^{-K_2t} \Vert u_0- v\Vert_{{}_{L^2}},
\end{equation*}
for some $K_2>0$.
\end{theorem}
\begin{proof}
Let $u=u(x,t)$ be the solution to \eqref{NonLBurg} and $w(x,t)=u(x,t)-v(x)$; we have
\begin{equation*}
	w_t=Q(\e u_x)_x-f(u)_x-Q(\e v_x)_x+f(v)_x.
\end{equation*}
By multiplying the latter equation for $w$ and integrating in $(-\ell,\ell)$ we obtain
\begin{equation}\label{eq:w-mono}
	\frac{1}{2}\frac{d}{dt}\|w\|^2_{{}_{L^2}}=\int_{-\ell}^{\ell}\left(f(u)-f(v)\right)w_x\,dx+\int_{-\ell}^\ell \left(Q(\e v_x)-Q(\e u_x)\right)w_x\,dx,
\end{equation}
where we used integration by parts.
The assumptions \eqref{eq:Qmono} on the function $Q$ imply that
\begin{equation*}
	\int_{-\ell}^\ell \left(Q(\e v_x)-Q(\e u_x)\right)w_x\,dx=\e\int_{-\ell}^\ell Q'(\xi)(v_x-u_x)w_x\,dx=-\e\int_{-\ell}^\ell Q'(\xi)w_x^2\,dx,
\end{equation*}
where $\xi=\xi(x,t)$ depends on $u_x(x,t)$ and $v_x(x)$.
In order to obtain a lower bound on $Q'(\xi)$, let us estimate the quantities $\|u_x(\cdot,t)\|_{{}_{L^\infty}}$ and $\|v_x\|_{{}_{L^\infty}}$.
From \eqref{eq:Fi(c)} it follows that
\begin{equation*}
	\frac{2\ell}\e\leq \frac{u_+-u_-}{Q^{-1}(m+c)}, \qquad \quad m=\min_{u\in[u_-,u_+]} f(u),
\end{equation*}
which implies
\begin{equation*}
	c\leq Q\left(\frac{\e(u_+-u_-)}{2\ell}\right)-m:=\bar Q,
\end{equation*}
and, as a consequence, 
\begin{equation*}
	\e\|v_x\|_{{}_{L^\infty}}\leq Q^{-1}(M+\bar Q), \qquad \quad M=\max_{u\in[u_-,u_+]} f(u).
\end{equation*}
From Proposition \ref{prop:aprioriux-Q}, it follows that
\begin{equation*}
	\e\Vert u_x(\cdot,t)\Vert_{{}_{L^\infty}}\leq Q^{-1}(\beta),
\end{equation*}
where $\beta$ is defined in \eqref{eq:smallness-Q}.
Therefore, denoting by $C:=\max\{M+\bar Q,\beta\}$, we get
\begin{equation*}
	Q'(\xi(x,t))\geq\min_{|s|\leq Q^{-1}(C)} Q'(s)=:c_0>0.
\end{equation*}
Using this estimate in \eqref{eq:w-mono}, we deduce
\begin{equation*}
	\frac{1}{2}\frac{d}{dt}\|w\|^2_{{}_{L^2}}\leq\int_{-\ell}^{\ell}\left(f(u)-f(v)\right)w_x\,dx-\e\, c_0\|w_x\|^2_{{}_{L^2}}.
\end{equation*}
We estimate the first term on the right hand side by using H\"older and Poincar\'e inequalities: 
\begin{align}
	\left\vert \int_{-\ell}^{\ell}\left(f(u)-f(v)\right)w_x\,dx\right\vert &\leq \left(  \sup_{u \in [u_-, u_+]} |f'(u)| \right) \int_{-\ell}^\ell \vert ww_x \vert \, dx
	\leq \left(  \sup_{u \in [u_-, u_+]} |f'(u)| \right)\|w\|_{{}_{L^2}}\|w_x\|_{{}_{L^2}} \notag \\
	&\leq c_p\left( \sup_{u \in [u_-, u_+]} |f'(u)| \right)\|w_x\|^2_{{}_{L^2}}, \label{eq:stimaf}
\end{align}
where $c_p=(2\ell/\pi)^2$. Hence
\begin{equation*}
	\frac{1}{2}\frac{d}{dt}\|w\|^2_{{}_{L^2}}+\left(\e c_0-c_p\sup_{u \in [u_-, u_+]} |f'(u)| \right)\|w_x\|^2_{{}_{L^2}}\leq0,
\end{equation*}
and we end up with
\begin{equation*}
	\frac{1}{2}\frac{d}{dt}\|w\|^2_{{}_{L^2}}+c_p^{-1}\left(\e\,c_0-c_p\sup_{u \in [u_-, u_+]} |f'(u)| \right)\|w\|^2_{{}_{L^2}}\leq0.
\end{equation*}
Choosing $K_1<\e c_0/c_p$ in \eqref{eq:ipo-stab-mono}, we deduce 
\begin{equation*}
	\frac{1}{2}\frac{d}{dt}\|w\|^2_{{}_{L^2}}+K_2\|w\|^2_{{}_{L^2}}\leq0,
\end{equation*}
and by integrating the latter inequality we obtain the thesis.
\end{proof}
Let us discuss the role of the assumption \eqref{eq:ipo-stab-mono} in the proof of Theorem \ref{thm:stab-mono}.
First of all, notice that, in the case of a linear function $f$, i.e. $f(u)=cu$, 
the first integral of the right hand side in \eqref{eq:w-mono} is zero and then we do not need such assumption.
If $f$ is not explicitly given, in order to exploit $K_1$ we need an estimate on the term $Q(\e u_x)-Q(\e v_x)$, 
that can be explicitly computed once $Q$ is given. 
For instance, when $Q(s)=\frac{s}{\sqrt{1+s^2}}$, we have, for all $a,b >0$
$$Q(a)-Q(b)= \frac{a^2-b^2}{\sqrt{1+a^2}\sqrt{1+b^2} \left( a\sqrt{1+b^2}+b\sqrt{1+a^2}\right)} \geq \frac{ a-b}{C^3},$$
where $C=\max\{\sqrt{1+a^2}, \sqrt{1+b^2}\}$.
Hence, the assumption \eqref{eq:ipo-stab-mono} reads
\begin{equation*}
	\max_{u \in [u_-, u_+]} |f'(u)| \leq\frac{c\, \e}{\ell^2}
\end{equation*}
for some $c>0$ independent of $\e$.
Therefore, if the flux function $f$ is fixed, assumption \eqref{eq:ipo-stab-mono} can be seen as a restriction on the boundary data $u_\pm$,
while if $f$ and $u_\pm$ are fixed, then the steady state is asymptotically stable if $\ell$ and $\e$ are chosen properly. 

\subsection{Non monotone and bounded dissipation flux function}
We here study the stability properties of the steady states to \eqref{NonLBurg} in the case of $Q$ given by 
\begin{equation}\label{caso2}
Q(s)=\displaystyle\frac{s}{1+s^2}.
\end{equation}
As in the previous case, we can state the following proposition giving an estimate for the $L^\infty$ norm of $u_x$.
\begin{proposition}\label{prop:aprioriux-nonmono}
Fix $T>0$ and let $u(\cdot,t)\in C^3(I)$ be a classical solution of the IBVP \eqref{NonLBurg} for $t\in[0,T]$ with $Q$ as in \eqref{caso2}. If 
\begin{equation}\label{eq:smallness-nonmono}
	\|u'_0\|_{{}_{L^\infty}}<\e^{-1} \qquad \mbox{ and } \qquad \left\|\frac{\e u'_0}{1+(\e u'_0)^2}\right\|_{{}_{L^\infty}}+2\Vert f(u_0)\Vert_{{}_{L^\infty}}\leq\gamma<\frac{1}2,
\end{equation}
then 
\begin{equation*}
	\Vert u_x(\cdot,t)\Vert_{{}_{L^\infty}}\leq\frac{2\gamma \e^{-1}}{1+\sqrt{1-4\gamma^2}}<\e^{-1},
\end{equation*}
for any $t\in[0,T]$.
\end{proposition}
\begin{remark}
{\rm
As  in Remark \ref{rem:ipoex-stab-mono}, we notice that assumptions \eqref{eq:smallness-nonmono} imply
\begin{equation*}
	M-m\leq2\Vert f(u_0)\Vert_{{}_{L^\infty}}<\frac12, \qquad \mbox{ and } \qquad   \left\|\frac{\e u'_0}{1+(\e u'_0)^2}\right\|_{{}_{L^\infty}}\leq\frac12-(M-m).
\end{equation*}
In particular, using $\|u'_0\|_{{}_{L^\infty}}<\e^{-1}$, the second estimate becomes
\begin{equation*}
	\|u'_0\|_{{}_{L^\infty}}\leq\frac{1-2(M-m)}{1+2\sqrt{(M-m)(1-(M-m))}}.
\end{equation*}
Assuming, for definiteness, that $u_->u_+$ and applying the latter estimate to
\begin{equation*}
	\frac{\e(u_--u_+)}{2\ell}\leq\|u'_0\|_{{}_{L^\infty}},
\end{equation*}
we obtain
\begin{equation*}
	\frac{2\ell}\e\geq\frac{1+2\sqrt{(M-m)(1-(M-m))}}{1-2(M-m)}(u_--u_+)\geq B,
\end{equation*}
where we used \eqref{eq:B-est}.
Therefore, assumptions \eqref{eq:smallness-nonmono} imply condition \eqref{eq:condexnonmono},
which guarantees the existence of a decreasing (smooth) steady state $v$ satisfying $\|v_x\|_{{}_{L^\infty}}<\e^{-1}$.
}
\end{remark}

\begin{proof}[Proof of Proposition \ref{prop:aprioriux-nonmono}]
We proceed in the same way as Proposition \ref{prop:aprioriux-Q}; we define the function
\begin{equation}\label{eq:z-nonmono}
	z(x,t):=\frac{\e u_x(x,t)}{1+\e^2u_x(x,t)^2}-f(u(x,t)),
\end{equation}
where $u$ is the classical solution of \eqref{NonLBurg} with $Q$ as in \eqref{caso2},
and we rewrite the equation for $u$ in the form
\begin{equation*}
 	u_t=z_x.
\end{equation*}
Differentiating equation \eqref{eq:z-nonmono} with respect to $t$, we deduce
\begin{equation}\label{eq:z_t-nonmono}
	z_t=\e\frac{1-\e^2u_x^2}{(1+\e^2u_x^2)^2}z_{xx}-f'(u)z_x.
\end{equation}
Reasoning as in the proof of Proposition \ref{prop:aprioriux-Q} and using \eqref{eq:smallness-nonmono} we can prove
that the equations in \eqref{NonLBurg} and \eqref{eq:z_t-nonmono} remain parabolic for $t\in[0,T]$;
by the maximum principle we thus have
\begin{equation}\label{aprioriu-nonmono} 
	\Vert u(\cdot,t)\Vert_{{}_{L^\infty}}\leq\Vert u_0\Vert_{{}_{L^\infty}},
\end{equation}
for any $t\in[0,T]$; moreover, the term $f'(u)$ is uniformly bounded and
\begin{equation*}
	\Vert z(\cdot,t)\Vert_{{}_{L^\infty}}\leq\Vert z(\cdot,0)\Vert_{{}_{L^\infty}}=\left\Vert \frac{\e u'_0}{1+(\e u'_0)^2}-f(u_0)\right\Vert_{{}_{L^\infty}}
\end{equation*}
holds for any $t\in[0,T]$, implying
\begin{equation*}
	\left\Vert \frac{\e u_x(\cdot,t)}{1+\e^2 u_x(\cdot,t)^2}\right\Vert_{{}_{L^\infty}}\leq\gamma<\frac{1}{2},
\end{equation*}
for any $t\in[0,T]$. Therefore, we can conclude 
\begin{equation*}
	\e\|u_x(\cdot,t)\|_{{}_{L^\infty}}\leq\frac{2\gamma}{1+\sqrt{1-4\gamma^2}}<2\gamma<1,
\end{equation*}
for any $t\in[0,T]$.
\end{proof}
In the next theorem we prove that the steady state $v$, solution to \eqref{eq:staz-nonmono-BV}, is asymptotically stable.
\begin{theorem}\label{thm:stab-nonmono}
Let $u$ be a classical solution to \eqref{NonLBurg}, with dissipative flux function $Q$ as in \eqref{caso2}, 
$u_->u_+$ and initial datum $u_0 \in C^3(I)$ satisfying \eqref{eq:smallness-nonmono}.
Then, there exists a positive constant $\Gamma_1$, (depending on $u_0$ and that can be explicitly computed) 
such that, if
\begin{equation}\label{eq:ipo-stab-nonmono}
	\max_{u \in [u_+, u_-]} |f'(u)| \leq \frac{\pi^2\e}{4\ell^2}  \, \Gamma_1,
\end{equation}
then 
\begin{equation*}
	\Vert u(\cdot,t)- v\Vert_{{}_{L^2}} \leq e^{-\Gamma_2t} \Vert u_0- v\Vert_{{}_{L^2}},
\end{equation*}
for some $\Gamma_2>0$.
\end{theorem}
\begin{proof}
Reasoning as in the proof of Theorem \ref{thm:stab-mono}, we find out that $w(x,t)=u(x,t)-v(x)$ satisfies
\begin{equation*}
	\frac{1}{2}\frac{d}{dt}\|w\|^2_{{}_{L^2}}=\int_{-\ell}^{\ell}\left(f(u)-f(v)\right)w_x\,dx+\int_{-\ell}^\ell \left(\frac{\e v_x}{1+\e^2 v^2_x}-\frac{\e u_x}{1+\e^2u^2_x}\right)w_x\,dx.
\end{equation*}
In this case, we use that for any $a,b\in\R$
\begin{equation*}
	\frac{a}{1+a^2}-\frac{b}{1+b^2}=\frac{(a-b)(1-ab)}{(1+a^2)(1+b^2)},
\end{equation*}
and we infer
\begin{equation*}
	\frac{1}{2}\frac{d}{dt}\|w\|^2_{{}_{L^2}}=\int_{-\ell}^{\ell}\left(f(u)-f(v)\right)w_x\,dx-\e\int_{-\ell}^\ell \left(\frac{1-\e^2v_xu_x}{(1+\e^2v_x^2)(1+\e^2u_x^2)}\right)w^2_x\,dx.
\end{equation*}
Hence, from \eqref{eq:stimaf} and Propositions \ref{prop:ex-nonmono}-\ref{prop:aprioriux-nonmono} one has
\begin{equation*}
	\frac{1}{2}\frac{d}{dt}\|w\|^2_{{}_{L^2}}\leq c_p\left( \sup_{u \in [u_-, u_+]} |f'(u)| \right)\|w_x\|^2_{{}_{L^2}}-\e\Gamma_0\|w_x\|^2_{{}_{L^2}},
\end{equation*}
where $c_p=(2\ell/\pi)^2$ and 
$$\Gamma_0=\frac{1-\e^2\|u_x(\cdot,t)\|_{{}_{L^\infty}}\|v_x\|_{{}_{L^\infty}}}{(1+\e^2\|u_x(\cdot,t)\|^2_{{}_{L^\infty}})(1+\e^2\|v_x\|^2_{{}_{L^\infty}})}>0.$$
Choosing $\Gamma_1<\Gamma_0$ in \eqref{eq:ipo-stab-nonmono}, we deduce 
\begin{equation*}
	\frac{1}{2}\frac{d}{dt}\|w\|^2_{{}_{L^2}}+\Gamma_2\|w\|^2_{{}_{L^2}}\leq0,
\end{equation*}
and by integrating the latter inequality we obtain the thesis.
\end{proof}

\begin{remark}
{\rm
Let us stress that the constants $\Gamma_0,\Gamma_1$ can be chosen independently on $\e$.
Indeed, from Proposition \ref{prop:ex-nonmono} it follows that $\|v_x\|_{{}_{L^\infty}}<\e^{-1}$, 
while by choosing $\gamma=1/4$ in \eqref{eq:smallness-nonmono}, we have $\|u_x(\cdot,t)\|_{{}_{L^\infty}}\leq1/\e(\sqrt3+2)$.
As a consequence, once $u_\pm$ and $f$ are fixed, it is possible to chose $\e$ so that condition \eqref{eq:ipo-stab-nonmono} holds. 
We also notice that, if $\e$ is very small, condition \eqref{eq:ipo-stab-nonmono} is very restrictive since, once $f$ is fixed, 
it gives stability only for {\it small} solutions (i.e. solutions connecting boundary data which are small with respect to $\e$). 
}
\end{remark}

\subsection{Monotone and unbounded dissipation flux function}
We finally study the stability of the monotone steady states of the initial boundary value problem \eqref{NonLBurg}, with dissipative flux function 
\begin{equation}\label{caso3}
	Q(s)=\dfrac{s}{\sqrt{1-s^2}}.
\end{equation}
As in the previous subsections, the first step is to establish an estimate for the $L^\infty$--norm of the space derivative of the solution $u$ to \eqref{NonLBurg}.
\begin{proposition}\label{prop:aprioriux-unbound}
Fix $T>0$ and let $u(\cdot,t)\in C^3(I)$ be a classical solution of the IBVP \eqref{NonLBurg} for $t\in[0,T]$ with $Q$ given by \eqref{caso3}.
If 
\begin{equation}\label{eq:smallness-unbound}
	\|u'_0\|_{{}_{L^\infty}}<\e^{-1},
\end{equation}
then 
\begin{equation*}
	\Vert u(\cdot,t)\Vert_{{}_{L^\infty}}\leq\Vert u_0\Vert_{{}_{L^\infty}}, \qquad \mbox{ and } \qquad \Vert u_x(\cdot,t)\Vert_{{}_{L^\infty}}<\e^{-1},
\end{equation*}
for any $t\in[0,T]$.
\end{proposition}
\begin{proof}
The function $z$, defined as
$$
	z(x,t):=\frac{\e u_x(x,t)}{\sqrt{1-\e^2 u_x(x,t)^2}}-f(u(x,t)), 
$$
satisfies the equation
\begin{equation*}
	z_t=\frac{\e z_{xx}}{\left(1-\e^2u_x^2\right)^{3/2}}-f'(u)z_x.
\end{equation*}
Set 
\begin{equation*}
	\left\|\frac{\e u'_0}{\sqrt{1-(\e u'_0)^2}}\right\|_{{}_{L^\infty}}+2\Vert f(u_0)\Vert_{{}_{L^\infty}}=\delta,
\end{equation*}
we have $\delta<+\infty$ (from \eqref{eq:smallness-unbound} and since $f\in C^2(\R)$).
Reasoning as in the proof of Propositions \ref{prop:aprioriux-Q} and \ref{prop:aprioriux-nonmono}, and by using the maximum principle we obtain
\begin{equation*}
	\Vert u(\cdot,t)\Vert_{{}_{L^\infty}}\leq\Vert u_0\Vert_{{}_{L^\infty}}, \qquad \mbox{ and } \qquad 
	\Vert z(\cdot,t)\Vert_{{}_{L^\infty}}\leq\Vert z(\cdot,0)\Vert_{{}_{L^\infty}},
\end{equation*}
for any $t\in[0,T]$.
In particular, we get 
\begin{equation*}
	\left\|\frac{\e u_x(\cdot,t)}{\sqrt{1-\e^2 u_x(\cdot,t)^2}}\right\|_{{}_{L^\infty}}\leq\Vert z(\cdot,0)\Vert_{{}_{L^\infty}}+\Vert z(\cdot,0)\Vert_{{}_{L^\infty}}\leq\delta,
\end{equation*}
and, as a trivial consequence
\begin{equation*}
	\e\|u_x(\cdot,t)\|_{{}_{L^\infty}}\leq\frac{\delta}{\sqrt{1+\delta^2}}<1.
\end{equation*}
The proof is complete.
\end{proof}

Let us stress that condition \eqref{eq:smallness-unbound} implies $2\ell>\e|u_+-u_-|$.
Thus, in this framework, the existence of smooth monotone steady states is guaranteed.
Before studying the stability of such steady states, we prove that,
when the assumptions of Proposition \ref{prop:aprioriux-unbound} are satisfied, the solution to the IBVP preserves the monotonicity of the initial datum.
\begin{proposition}\label{prop:crescente}
Let $u=u(x,t)$ be a classical solution of \eqref{NonLBurg}, 
with $Q$ as in \eqref{caso3} and with monotone increasing (decreasing) initial datum $u_0 \in C^3(I)$ satisfying \eqref{eq:smallness-unbound}. 
Then, for every $t>0$, $u(\cdot, t)$ is monotone increasing (decreasing). 
\end{proposition}
\begin{proof}
Consider the case of an increasing initial datum $u_0$.
Proposition \ref{prop:aprioriux-unbound} implies that
\begin{equation}\label{aprioriu-monotone}
	u_-\leq u(x,t)\leq u_+, \qquad \qquad \forall\, x\in [-\ell,\ell], \, t\geq0.
\end{equation}
By differentiating with respect to $x$ the equation for $u$, we obtain that $y=u_x$ solves
\begin{equation}
	y_t = \frac{\e y_{xx}}{(1-\e^2y^2)^{3/2}}+\frac{3\e^3y y_x^2}{(1+\e^2y^2)^{5/2}}-f'(u)y_x-f''(u)y^2.
\end{equation}
The latter equation is parabolic and both $y=0$ and $y=u_x$ are solutions; 
since $y(x,0)\geq0$ by assumption and $y(\pm\ell, t)=u_x(\pm \ell, t) \geq 0$ for any $t\geq0$ (otherwise \eqref{aprioriu-monotone} would be violated),
from the comparison principle \cite[Theorem 9.7]{Lie}, it follows that $y(x, t)\geq0$ for every $x\in[-\ell, \ell]$ and $t > 0$, namely $u(\cdot, t)$ is increasing for all $t>0$. 
The case of a decreasing initial datum $u_0$ can be treated similarly. 
\end{proof} 

We now have all the tools to prove the stability of the increasing steady state $v$, solution to \eqref{eq:staz-unbound-BV}.
\begin{theorem}\label{thm:stab-unbound}
Let $u$ be a classical solution to \eqref{NonLBurg}, with dissipative flux function $Q$ as in \eqref{caso3}, 
$u_+>u_-$ and a monotone increasing initial datum $u_0 \in C^3(I)$ satisfying \eqref{eq:smallness-unbound}.
If
\begin{equation}\label{eq:ipo-stab-unbound}
	\max_{u \in [u_+, u_-]} |f'(u)|<\frac{\pi^2\e}{4\ell^2},
\end{equation}
then 
\begin{equation*}
	\Vert u(\cdot,t)- v\Vert_{{}_{L^2}} \leq e^{-\Lambda t} \Vert u_0- v\Vert_{{}_{L^2}},
\end{equation*}
for some $\Lambda>0$.
\end{theorem}
\begin{proof}
The function $w(x,t)=u(x,t)-v(x)$ satisfies
\begin{equation*}
	\frac{1}{2}\frac{d}{dt}\|w\|^2_{{}_{L^2}}=\int_{-\ell}^{\ell}\left(f(u)-f(v)\right)w_x\,dx+\int_{-\ell}^\ell \left(\frac{\e v_x}{\sqrt{1-\e^2v^2_x}}-\frac{\e u_x}{\sqrt{1-\e^2u^2_x}}\right)w_x\,dx.
\end{equation*}
Since, for any $|a|<1$, $|b|<1$,
\begin{equation*}
	\frac{a}{\sqrt{1-a^2}}-\frac{b}{\sqrt{1-b^2}}=\frac{(a-b)(a+b)}{b(1-a^2)\sqrt{1-b^2}+a(1-b^2)\sqrt{1-a^2}},
\end{equation*}
we have
\begin{align*}
	\frac{1}{2}\frac{d}{dt}\|w\|^2_{{}_{L^2}}=&\int_{-\ell}^{\ell}\left(f(u)-f(v)\right)w_x\,dx\\
	&\quad -\e\int_{-\ell}^\ell \left(\frac{v_x+u_x}{u_x(1-\e^2v_x^2)\sqrt{1-\e^2u_x^2}+v_x(1-\e^2u_x^2)\sqrt{1-\e^2v_x^2}}\right)w^2_x\,dx.
\end{align*}
Using \eqref{eq:stimaf} and since $v_x\in(0,\e^{-1})$ and $u_x\in(0,\e^{-1})$ for any $(x,t)$ 
(see Propositions \ref{prop:ex-unbound}, \ref{prop:aprioriux-unbound} and \ref{prop:crescente}), we infer
\begin{equation*}
	\frac{1}{2}\frac{d}{dt}\|w\|^2_{{}_{L^2}}\leq \frac{4\ell^2}{\pi^2}\left( \sup_{u \in [u_-, u_+]} |f'(u)| \right)\|w_x\|^2_{{}_{L^2}}-\e\|w_x\|^2_{{}_{L^2}}.
\end{equation*}
Using the assumption \eqref{eq:ipo-stab-unbound}, we deduce that there exists $\Lambda>0$ such that
\begin{equation*}
	\frac{1}{2}\frac{d}{dt}\|w\|^2_{{}_{L^2}}+\Lambda\|w\|^2_{{}_{L^2}}\leq0,
\end{equation*}
and this concludes the proof.
\end{proof}
We briefly mention that also in this case  if $\e$ is very small, condition \eqref{eq:ipo-stab-unbound} implies a smallness (with respect to $\e$) condition on the boundary data. 
As we will see in the next section, when $\e\ll1$ and the boundary data are large enough, we will face the phenomenon of metastability, 
i.e. slow convergence towards the asymptotic limit.  

\section{Metastability for the small dissipation limit}\label{sec4}
In this section, we analyze the occurrence of a metastable dynamics for the IBVP \eqref{NonLBurg}; 
in general, a metastable behavior appears when the solution to a given evolution PDE of the form
\begin{equation}\label{sto}
	 u_t = \mathcal P^\e[u],
\end{equation}
(where $\mathcal P^\e$ is a nonlinear differential operator that is singular with respect to the parameter $\e$)
approaches his stable ({\it metastable}) steady state in an exponentially long time interval, usually of the order $\mathcal{O}(\exp\left(c/\e\right))$. 
There is a huge literature investigating such phenomenon for different evolution PDEs and by means of different techniques: 
far from being exhaustive see, among others, \cite{ABF, CarrPego89, FLM1, FLM2, OttoRezn06, Str13} for Allen-Cahn and Cahn-Hilliard like equations, 
\cite{LafoOMal94, ReynWard95, Str14, SunWard99} for viscous shock problems, and the references therein.

Before showing the theory in the case of nonlinear diffusions, let us briefly recall what happens when considering the linear viscous Burgers equation
\begin{equation}\label{eq:visc-burg}
u_t+f(u)_x=\e u_{xx},
\end{equation}
with boundary data and flux function $f$ satisfying the following conditions 
\begin{equation}\label{eq:f-burgers}
 	f(u_+)=f(u_-), \qquad f'(u_+)<0<f'(u_-), \qquad f''(u)\geq c_0>0, \quad \forall\,u\in\R.
\end{equation} 
It is well know that conditions \eqref{eq:f-burgers} imply that there exist infinitely many stationary solutions to the hyperbolic conservation law
\begin{equation}\label{eq:hyp-cons}
	u_t+f(u)_x=0, \qquad \qquad u(\pm\ell,t)=u_\pm,
\end{equation}
satisfying both Rankine-Hugoniot and entropy conditions.
In particular, any step function of the form
\begin{equation*}
	u^{\xi}(x)=\begin{cases}
			u_-, \qquad \qquad x\in(-\ell,\xi),\\
			u_+, \qquad \qquad x\in(\xi,\ell),
		\end{cases}
\end{equation*}
where $\xi\in(-\ell,\ell)$, is a stationary solution to \eqref{eq:hyp-cons},
which satisfies both the Rankine-Hugoniot and the entropy conditions.
When a viscous term $\e u_{xx}$ is added in \eqref{eq:hyp-cons}, the number of stationary solutions drastically reduces, 
and we have a unique solution to the problem
\begin{equation}\label{eq:staz-linear}
	f(u)_x=\e u_{xx}, \qquad \qquad u(\pm\ell,t)=u_\pm,
\end{equation}
with $f$ and $u_\pm$ satisfying \eqref{eq:f-burgers}.
On the other hand, the steady state \eqref{eq:staz-linear} is asymptotically stable for equation \eqref{eq:visc-cons-law} as $t\to+\infty$,
namely, if $u^\e$ is the solution to the IBVP \eqref{NonLBurg} with $Q(s)=\e s$ and $f,u_\pm$ satisfy \eqref{eq:f-burgers}, then
\begin{equation*}
	\lim_{\e\to0^+}\left(\lim_{t\to+\infty} u^\e(x,t)\right)=\lim_{\e\to0^+}u^\e_\infty(x)=\begin{cases}
			u_-, \qquad \qquad x\in(-\ell,0),\\
			u_+, \qquad \qquad x\in(0,\ell),
		\end{cases}
\end{equation*}
where $u^\e_\infty$ is the unique solution to \eqref{eq:staz-linear}, and the limit function has a jump at $0$ because we are considering the symmetric interval $[-\ell,\ell]$.
On the contrary, exchanging the order of the limits, we have
 \begin{equation*}
	\lim_{t\to+\infty}\left(\lim_{\e\to0^+} u^\e(x,t)\right)=\lim_{t\to+\infty}u^0(x)=\begin{cases}
			u_-, \qquad \qquad x\in(-\ell,\xi),\\
			u_+, \qquad \qquad x\in(\xi,\ell),
		\end{cases}
\end{equation*}
where $u^0$ is a solution to \eqref{eq:hyp-cons}.
Hence, the two limits are not interchangeable.

Many papers have been devoted to the study of the dynamics of the solutions to the viscous Burgers equation \eqref{eq:visc-burg}
in the small viscosity limit, and it is well known that, when assumptions \eqref{eq:f-burgers} are satisfied, the solutions exhibit the phenomenon of \emph{metastability}.
In particular, the solution to \eqref{eq:visc-burg} reaches its asymptotic limit after an \emph{exponentially long time},
namely a time $T_\e=\mathcal{O}\left(\exp(c/\e)\right)$, as $\e\to0^+$, and we thus have a {slow} convergence to the steady state.

\vskip0.2cm
Motivated by the behavior of its linear counterpart \eqref{eq:visc-burg}, 
we here consider the IBVP \eqref{NonLBurg} and we assume the flux function to satisfy conditions \eqref{eq:f-burgers}; 
we stress that such conditions imply, among others, that  $u_->u_+$, and, as we will see, 
this is a necessary condition for the appearance of a metastable behavior (the same happens in the linear case).

To prove that the same pattern described for the linear equation appears when considering nonlinear diffusions as \eqref{eq:Qnonmonotone} and \eqref{eq:Qnonbound}, 
we mean to adapt the strategy first developed in \cite{MS} 
and subsequently used in \cite{FGS} to study metastability for a viscous conservation law with a mean curvature type operator \eqref{MC};
in particular, the three main steps of such strategy, that we briefly recall here for the reader's convenience, are the following:
\smallbreak
\begin{itemize}
\item {\bf Step I. Construction of a one family of approximate steady states.} 
The first step is the construction of a one-parameter family {\it  of approximated steady states} $\{ U^\e(x;\xi) \}_{\xi}$, whose generic element is built in a way such that 
\begin{equation}\label{step1}
	|\langle \psi(\cdot ), \mathcal P^\e[U^\e(\cdot;\xi)]\rangle| \leq \Omega^\e(\xi)|\psi|_{L^\infty}, \quad \forall \, \psi \in C^{\infty}(I), \, \forall \, \xi \in I,
\end{equation}
where $\Omega^\e(\xi)$ is a smooth positive function such that $\Omega^\e \to 0$ for $\e \to 0$ and  $\mathcal{P}^\e$ is the operator on the right-hand side of \eqref{sto}. 
Hence $U^\e$ is constructed so that the quantity $ P^\e[U^\e]$ (which is exactly zero if $U^\e$ is the exact steady state) 
is small with respect to $\e$, and the error made is measured by $\Omega^\e$.

\smallbreak
\item {\bf Step II. Linearization.}
The second step is the linearization of the original system \eqref{sto} around an element of the family $\{U^\e(x;\xi)\}_\xi$, i.e. one  looks for a solution of the form
\begin{equation}\label{decu}
	u(x,t)= U^\e(x;\xi(t))+v(x,t),
\end{equation}
with $\xi=\xi(t)\in I$  and the perturbation ${ v}={v}(x,t)\in L^2(I)$ to be determined. 
The idea is to suppose the parameter $\xi$ to depend on time, so that to think at $\xi$ as the position of the internal interface of the solution $u$; 
in particular, its evolution describes the asymptotic convergence of the interface solution towards the equilibrium. 
\smallbreak
\item {\bf Step III. Spectral properties.} 
Last step is a spectral analysis of the linearized operator around $U^\e$, named here $\mathcal L^\e$; 
in order to derive an equation for the perturbation $v$, to be coupled with an equation of motion for the parameter $\xi$, 
we have to check that $\mathcal L^\e$ has a discrete spectrum composed by real and semi-simple eigenvalues $\{ \lambda_k^\varepsilon(\xi) \}_{k \in \N}$ such that  
$$\lim_{\e \to 0}\lambda_1^\varepsilon(\xi) = 0 \qquad \mbox{and \qquad}  \lambda^\varepsilon_k(\xi) \leq -C \quad \mbox{ for all} \ k \geq 2,$$  
for any $\xi\in I$, and for some constant $C>0$ independent of $k$, $\varepsilon$ and $\xi$. 
\end{itemize}

We refer the reader to \cite[Section 5]{FGS}, for a complete discussion on the strategy and on the assumptions. 
In order to apply such a strategy we need to consider an explicit expression of $Q$, 
this being the reason why we consider here the cases \eqref{eq:Qnonmonotone} and \eqref{eq:Qnonbound};
we stress again that the case \eqref{MC} has been already fully treated in \cite{FGS}.
\subsection{Construction of the one-parameter family - nonmonotone case}
In this section we mean at verify {\bf Step I} of the aforementioned strategy when $Q$ is given by \eqref{eq:Qnonmonotone},
i.e we consider the equation 
\begin{equation}\label{eq:nonmono-eps}
	u_t+ f(u)_x=\left(\frac{\e u_x}{1+\e^2 u_x^2}\right)_x, \qquad x\in I, \,t > 0,
\end{equation}
subject to boundary conditions
\begin{equation}\label{eq:boundary}
	u(\pm\ell,t)=u_\pm, \qquad \quad t\geq0,
\end{equation}
and initial datum
\begin{equation}\label{eq:initial}
	u(x,0)=u_0(x), \qquad \quad x\in I.
\end{equation}
We recall that the flux function $f\in C^2(\R)$ and the boundary data $u_\pm$ are assumed to satisfy \eqref{eq:f-burgers}.
Notice that assumptions \eqref{eq:f-burgers} imply the existence of a unique $\bar u\in(u_+,u_-)$ such that $f'(\bar u)=0$;
in the rest of the paper, without loss of generality we assume that $\bar u=0$ and that $f(0)=0$.
As a consequence, we have $u_+<0<u_-$.
Hence, from now on, we assume that
\begin{equation}\label{eq:f-metastability}
 	f(0)=f'(0)=0, \qquad f(u_+)=f(u_-), \qquad f'(u_+)<0<f'(u_-), \qquad \;f''(u)\geq c_0>0,
\end{equation}
for all $u \in \R$. 
Observe that from \eqref{eq:f-metastability}, it follows
\begin{equation*}
	m=\min_{u\in[u_+,u_-]}f(u)=0, \qquad \mbox{ and } \qquad M=\max_{u\in[u_+,u_-]}f(u)=f(u_\pm),
\end{equation*}
and in order to have existence of a unique smooth decreasing stationary solutions of \eqref{eq:nonmono-eps} 
(see Proposition \ref{prop:ex-nonmono}) we have to impose the condition
\begin{equation}\label{eq:condnonmono-eps}
	f(u_\pm)<\frac{1}{2}.
\end{equation}
The goal of this subsection is to construct a family of functions $U^\e(\cdot;\xi)$ approximating steady states
of the initial boundary value problem \eqref{eq:nonmono-eps}-\eqref{eq:boundary}-\eqref{eq:initial}.

\begin{proposition}\label{prop:metastablestates-nonmono}
Let $u_\pm$, $f$ be such that \eqref{eq:f-metastability} and \eqref{eq:condnonmono-eps} hold, and for $u\in H^2(I)$ denote by
\begin{equation*}
 	\mathcal P^\e[u]:=\left(\frac{\e u_x}{1+\e^2u_x^2}\right)_x-f(u)_x.
\end{equation*}
Then there exists a family of functions $\left\{U^{\varepsilon}(\cdot;\xi)\right\}_{\xi\in I}\in H^1(I)$ satisfying \eqref{eq:boundary} such that
\begin{equation}\label{eq:H1-nonmono}
	|\langle \psi,\mathcal P^\e[U^\e(\cdot;\xi)]\rangle|\leq\Omega^\e(\xi)\|\psi\|_{{}_{L^\infty}}, \qquad \qquad \forall \, \psi \in C^{\infty}(I), \, \forall \, \xi \in I,
\end{equation}
where, for $\e$ small, the function $\Omega^\e(\xi)$ satisfies
\begin{equation}\label{eq:Omega-nonmono}
	\Omega^\e(\xi)\leq c_1 \exp(-c_2(\ell-|\xi|)/\e), \qquad \qquad \xi\in I,
\end{equation}
and $c_1, c_2>0$ depend on $u_\pm$ and $f(u_\pm)$.
Moreover, there exists $\bar{\xi}\in I$ such that $\Omega^\e(\bar \xi) \equiv 0$.
\end{proposition}

\begin{proof}
The family of functions $\left\{U^{\varepsilon}(\cdot;\xi)\right\}_{\xi\in I}$ is constructed by matching together two stationary solutions to \eqref{eq:unbound-eps}.
Precisely, denoting by $U^\e_-(\cdot;\xi)$ the solution to the boundary value problem
\begin{equation}\label{eq:U-nonmono}
	\mathcal P^\e[U^\e_-]=0 \qquad \mbox{ in } \, (-\ell,\xi), \qquad \qquad U^\e_-(-\ell;\xi)=u_-, \qquad U^\e_-(\xi;\xi)=0,
\end{equation} 
with $({U^\e_-})_x\in(-1,0)$ and by $U^\e_+ (\cdot;\xi)$ the solution to
\begin{equation}\label{eq:U+nonmono}
	\mathcal P^\e[U^\e_+]=0 \qquad \mbox{ in } \, (\xi,\ell), \qquad \qquad U^\e_+(\xi;\xi)=0, \qquad U^\e_+(\ell;\xi)=u_+,
\end{equation} 
with $({U^\e_+})_x\in(-1,0)$, we define the generic element of the family $\{U^\e\}_{\xi \in I}$ as
\begin{equation}\label{eq:Ueps}
	U^{\varepsilon}(x;\xi)=\left\{\begin{aligned}
		&U^\e_-(x;\xi), &\qquad & x \in (-\ell,\xi), \\
		&U^\e_+ (x;\xi), &\qquad & x \in (\xi,\ell).
           \end{aligned}\right.
\end{equation}
The existence of the solutions $U^\e_-(\cdot;\xi)$ and $U^\e_+(\cdot;\xi)$ can be proved as in Section \ref{sec:ex-nonmono}:
the function $U^\e_-$ is implicitly given by 
\begin{equation*}
	\int_{U^\e_-(x;\xi)}^{u_-}\frac{1+\sqrt{1-4\left(c-f(s)\right)^2}}{2\left(c-f(s)\right)}\,ds=\frac{\ell+x}\e,
\end{equation*}
where $c\in(f(u_\pm),1/2)$.
By imposing $U^\e_-(\xi;\xi)=0$, we get the condition 
\begin{equation*}
	\Psi(c):=\int_{0}^{u_-}\frac{1+\sqrt{1-4\left(c-f(s)\right)^2}}{2\left(c-f(s)\right)}\,ds=\frac{\ell+\xi}\e.
\end{equation*}
Since $\Psi$ is a decreasing function satisfying  
\begin{equation*}
	\lim_{c\to f(u_\pm)}\Psi(c)= +\infty,  \qquad \quad\lim_{c\to1/2}\Psi(c)=\int_{0}^{u_-}\frac{1+2\sqrt{f(s)(1-f(s))}}{1-2f(s)}\,ds=:B,
\end{equation*}
there exists a unique solution to \eqref{eq:U-nonmono} if and only if $\xi>-\ell+\e B$.
Thus, for any $\xi\in(-\ell,\ell)$ we can choose $\e$  small enough so that the existence of a unique solution to \eqref{eq:U-nonmono} is guaranteed.
Similarly, we obtain the existence of the unique solution $U^\e_+ (\cdot;\xi)$ to \eqref{eq:U+nonmono}.
In conclusion, for any $\xi\in(-\ell,\ell)$ we can choose $\e$ sufficiently small so that the function $U^{\varepsilon}(\cdot;\xi)$ in \eqref{eq:Ueps} is well-defined, and we have
\begin{equation}\label{eq:Upm-implicit-nonmono}
	\int_{0}^{u_\pm}\frac{1+\sqrt{1-4\left(\k_\pm-f(s)\right)^2}}{2\left(\k_\pm-f(s)\right)}\,ds=\frac{\xi\mp\ell}\e,
\end{equation}
for some $\kappa_\pm=\kappa_\pm(\xi)\in(f(u_\pm),1/2)$. In particular, the parameter $\xi$ (which is exactly the gluing point among $U^\e_-$ and $U_+^\e$) can be thought as the position of the internal interface of the solution.

We now verify \eqref{eq:H1-nonmono}.
A straightforward computation shows that
\begin{equation*}
	\langle \psi,\mathcal P^\e[U^\e(\cdot;\xi)] \rangle= \psi(\xi)(\k_-(\xi)-\k_+(\xi)) \quad \mbox{ for any} \ \ \psi \in C^1(I),
\end{equation*}
so that, in the distributional sense
\begin{equation*}
	\mathcal P^\e[U^\e(\cdot; \xi)]=(\k_-(\xi)-\k_+(\xi)) \delta_{x=\xi},
\end{equation*}
and we can choose $\Omega^\e=\k_--\k_+$ in \eqref{eq:H1-unbound}.
In order to obtain the estimate \eqref{eq:Omega-nonmono}, we need to evaluate $\k_--\k_+$.
To approximately compute such difference, we observe that, in view of the convexity assumed in \eqref{eq:f-metastability} for the flux function $f$, the following bounds hold: 
\begin{equation}\label{boundF}
	\begin{aligned}
	&f(u_\pm)+f'(u_+)(u-u_+) \leq f(u) \leq \frac{f(u_\pm)}{u_+} u,  \qquad & u \in [u_+,0],\\
	& f(u_\pm)-f'(u_-)(u_- -u) \leq f(u) \leq \frac{f(u_\pm)}{u_-} u, \qquad & u \in [0, u_-]. 
	\end{aligned}
\end{equation}
Using \eqref{eq:Upm-implicit-nonmono}, we get
\begin{equation*}
	\frac{\xi-\ell}\e=\int_0^{u_+} \frac{1+\sqrt{1-4\left(\k_+-f(s)\right)^2}}{2\left(\k_+-f(s)\right)} \, ds \leq\frac12\int_{u_+}^0 \frac{ds}{f(s)-\k_+},
\end{equation*}
where we used that $\k_+ - f(s)>0$ for any $s\in(u_+,0)$.
Hence, by using the lower bound in the first estimate of \eqref{boundF}, we infer
\begin{align*}
	\frac{2(\xi-\ell)}{\e} \leq \int_{u_+}^0 \frac{ds}{f(s)-\k_+} &\leq \int_{u_+}^0 \frac{ds}{ f(u_\pm)+f'(u_+)(s-u_+) -\k_+}\\
	& =\frac{1}{f'(u_+)}\log\left(\k_+-f(u_\pm)-f'(u_+)(s-u_+)\right) \Big|^{0}_{u_+}, \\
	&=\frac{1}{f'(u_+)}\log\left(1+\frac{u_+f'(u_+)}{\k_+-f(u_\pm)}\right),
\end{align*}
that is, since $f'(u_+)<0$
\begin{equation*}
	\exp\left(2f'(u_+)(\xi-\ell)/\e\right)\geq1+\frac{u_+f'(u_+)}{\k_+-f(u_\pm)} \quad  \Longrightarrow \quad \k_+-f(u_\pm)\geq \frac{u_+f'(u_+)}{\exp\left(2f'(u_+)(\xi-\ell)/\e\right)-1}.
\end{equation*}
On the other side, from \eqref{eq:Upm-implicit-nonmono} one has
\begin{equation*}
	\frac{\xi-\ell}\e= \int_0^{u_+} \frac{1+\sqrt{1-4\left(\k_+-f(s)\right)^2}}{2\left(\k_+-f(s)\right)} \, ds\geq\int_{u_+}^0 \frac{ds}{\frac{f(u_\pm)}{u_+} s-\k_+},
\end{equation*}
where we used the upper bound in the first estimate of \eqref{boundF}.
By doing similar computations as above, we end up with
\begin{equation*}
	\exp\left(f(u_\pm)(\xi-\ell)/(\e\,u_+)\right)\leq1+\frac{f(u_\pm)}{k_+-f(u_\pm)},
\end{equation*}
implying
\begin{equation*}
	\k_+-f(u_\pm)\leq\frac{f(u_\pm)}{\exp\left(f(u_\pm)(\xi-\ell)/\e\,u_+\right)-1}.
\end{equation*}
As concerning $\k_-$, again from \eqref{eq:Upm-implicit-nonmono} we have
$$
	\frac{\xi+\ell}\e=\int_0^{u_-}\frac{1+\sqrt{1-4\left(\k_--f(s)\right)^2}}{2\left(\k_--f(s)\right)}\,ds \leq\int_0^{u_-} \frac{ds}{\k_--f(s)}\leq\int_{0}^{u_-} \frac{ds}{\k_- -\frac{f(u_\pm)}{u_-} s}, 
$$
and 
\begin{align*}
	\frac{\xi+\ell}\e=\int_0^{u_-} \frac{1+\sqrt{1-4\left(\k_--f(s)\right)^2}}{2\left(\k_--f(s)\right)}\,ds &\geq\frac12\int_0^{u_-} \frac{ds}{\k_--f(s)}\\
	&\geq\frac12\int_{0}^{u_-} \frac{ds}{\k_--f(u_\pm)+f'(u_-)(u_--s)};
\end{align*}
we can thus proceed as before to obtain upper and lower bounds on the difference $\k_- - f(u_\pm)$. 
In conclusion, collecting all the computations, we obtain
\begin{equation}\label{eq:kpm-nonmono}
	\begin{aligned}
	\frac{u_+f'(u_+)}{\exp\left(2f'(u_+)(\xi-\ell)/\e\right)-1}\leq \k_+ - f(u_\pm) &\leq\frac{f(u_\pm)}{\exp\left(f(u_\pm)(\xi-\ell)/\e\,u_+\right)-1}, \\
	\frac{u_-f'(u_-)}{\exp\left(2f'(u_-)(\xi+\ell)/\e\right)-1}\leq \k_- - f(u_\pm) & \leq \frac{f(u_\pm)}{\exp\left({f(u_\pm)(\xi+\ell)/{\e \, u_-}}\right)-1}.
	\end{aligned}
\end{equation}
Thanks to \eqref{eq:kpm-nonmono}, we deduce the estimate \eqref{eq:Omega-nonmono} for $\Omega^\e=\k_--\k_+$.

Finally, denoting by $g(\xi):= \k_-(\xi)-\k_+(\xi)$, for $\xi\in I$, we observe that $g$ is a monotone decreasing function satisfying
\begin{equation*}
	\lim_{\xi\to-\ell+\e B}g(\xi)>0, \qquad \mbox{ and } \qquad \lim_{\xi\to\ell-\e B}g(\xi)<0.
\end{equation*}
It follows that there exists a unique value $\bar\xi$ such that $g(\bar \xi)=0$ and 
$U^\e(\cdot;\bar \xi)$ is the unique steady state of the boundary value problem \eqref{eq:nonmono-eps}-\eqref{eq:boundary}. 
\end{proof}
\begin{remark}\label{rem4.2}
{\rm
The positive constants $c_1$ and $c_2$ in \eqref{eq:Omega-nonmono} depend on $u_\pm$ and $f(u_\pm)$. 
For instance, in case of the Burgers flux function $f(u)=u^2/2$, one has $u_-=-u_+=u_*>0$; 
therefore, using \eqref{eq:kpm-nonmono}, if
\begin{equation*}
	\lim_{\e\to0}\frac{u_*}{\e}=\mp\infty,
\end{equation*}
then $\Omega^\e=\k_--\k_+$ is exponentially small for $\e\to0$, uniformly in any compact subset of $I$.
In particular, estimate \eqref{eq:Omega-nonmono} reads
\begin{equation}\label{eq:Omega-nonmono-burg}
	\Omega^\e(\xi)\leq Cu_*^2\exp(-u_*(\ell-|\xi|)/\e), \qquad \qquad \xi\in I,
\end{equation}
for some positive constant $C>0$ independent on $\e$ and $u_*$.
}
\end{remark}

\subsection{Construction of the one-parameter family - unbound case}
We here consider the dissipation flux function given by \eqref{eq:Qnonbound} with $\kappa=1$, i.e. we consider the equation 
\begin{equation}\label{eq:unbound-eps}
	u_t+ f(u)_x=\left(\frac{\e u_x}{\sqrt{1-\e^2u_x^2}}\right)_x, \qquad x\in I, \,t > 0,
\end{equation}
subject to boundary and initial conditions \eqref{eq:boundary}-\eqref{eq:initial}, 
with flux function $f\in C^2(\R)$ and boundary data $u_\pm$ satisfying \eqref{eq:f-metastability}.
As before, we want to construct a family  $U^\e(\cdot;\xi)$ of functions approximating the steady states
of the initial boundary value problem \eqref{eq:unbound-eps}-\eqref{eq:boundary}-\eqref{eq:initial}.

\begin{proposition}\label{prop:metastablestates-unbounded}
Assume $u_\pm$ and $f$ be such that  $u_-  - u_+ < \frac{2\ell}\e$ for any $\e$ sufficiently small and \eqref{eq:f-metastability} holds.
For $u\in H^2(I)$ denote by 
\begin{equation*}
 	\mathcal P^\e[u]:=\left(\frac{\e u_x}{\sqrt{1-\e^2u_x^2}}\right)_x-f(u)_x.
\end{equation*}
Then there exists a family of functions $\left\{U^{\varepsilon}(\cdot;\xi)\right\}_{\xi\in I}\in H^1(I)$ satisfying \eqref{eq:boundary} such that
\begin{equation}\label{eq:H1-unbound}
	|\langle \psi,\mathcal P^\e[U^\e(\cdot;\xi)]\rangle|\leq\Omega^\e(\xi)\|\psi\|_{{}_{L^\infty}}, \qquad \qquad \forall \, \psi \in C^{\infty}(I), \, \forall \, \xi \in I,
\end{equation}
where, for $\e$ small, the function $\Omega^\e(\xi)$ satisfies \eqref{eq:Omega-nonmono}.
Moreover, there exists $\bar{\xi}\in I$ such that $\Omega^\e(\bar \xi) \equiv 0$. 
\end{proposition}

\begin{proof}
As in Proposition \ref{prop:metastablestates-nonmono}, the family of functions $\left\{U^{\varepsilon}(\cdot;\xi)\right\}_{\xi\in I}$ 
is constructed by matching together at $x=\xi$ (interface position) two stationary solutions $U^\e_{\pm}$ to \eqref{eq:unbound-eps};
the generic element of the family $\{U^\e\}_{\xi \in I}$ is defined as in \eqref{eq:Ueps} and also in this case $U_x^\e \in (-1,0)$ for all $x \neq \xi$. 
The function $U^\e_-$ is implicitly given by 
\begin{equation*}
	\int_{U^\e_-(x;\xi)}^{u_-}\frac{\sqrt{1+(c-f(s))^2}}{c-f(s)}\,ds=\frac{\ell+x}\e,
\end{equation*}
where $c\in(f(u_\pm),+\infty)$.
Imposing $U^\e_-(\xi;\xi)=0$, we get the condition 
\begin{equation*}
	\Psi(c):=\int_{0}^{u_-}\frac{\sqrt{1+(c-f(s))^2}}{c-f(s)}\,ds=\frac{\ell+\xi}\e.
\end{equation*}
Since $\Psi$ is a decreasing function satisfying  
\begin{equation*}
	\lim_{c\to f(u_\pm)}\Psi(c)= +\infty,  \qquad \quad\lim_{c\to+\infty}\Psi(c)=u_-,
\end{equation*}
we have the  existence of a unique $U^\e_- (\cdot;\xi)$ if and only if $\xi>-\ell+\e u_-$.
Similarly, we obtain the existence of  $U^\e_+ (\cdot;\xi)$  if and only if $\xi<\ell+\e u_+$.
Therefore, the function $U^{\varepsilon}(\cdot;\xi)$ in \eqref{eq:Ueps} is well defined for any $\xi\in(-\ell,\ell)$ if we choose $\e$ sufficiently small 
and we have
\begin{equation}\label{eq:Upm-implicit-unbound}
	\int_{0}^{u_\pm}\frac{\sqrt{1+(\kappa_\pm-f(s))^2}}{\kappa_\pm-f(s)}\,ds=\frac{\xi\mp\ell}\e,
\end{equation}
for some $\kappa_\pm=\kappa_\pm(\xi)>f(u_\pm)$.
We now verify \eqref{eq:H1-unbound}; as before
\begin{equation*}
	\mathcal P^\e[U^\e(\cdot; \xi)]=(\k_-(\xi)-\k_+(\xi)) \delta_{x=\xi},
\end{equation*}
and we can choose $\Omega^\e=\k_--\k_+$ in \eqref{eq:H1-unbound}.
Let us evaluate the difference $\k_--\k_+$ by using the estimates \eqref{boundF} and \eqref{eq:Upm-implicit-unbound}. We get
\begin{equation*}
	\frac{\xi-\ell}\e=\int_0^{u_+} \frac{\sqrt{1+(\k_+-f(s))^2}}{\k_+-f(s)} \, ds \leq\int_{u_+}^0 \frac{ds}{f(s)-\k_+}\leq\int_{u_+}^0 \frac{ds}{ f(u_\pm)+f'(u_+)(s-u_+) -\k_+},
\end{equation*}
where we used that $\k_+ - f(s)>0$ for any $s\in(u_+,0)$ and the lower bound in the first estimate of \eqref{boundF}.
Therefore, we end up with
\begin{equation*}
	\k_+-f(u_\pm)\geq \frac{u_+f'(u_+)}{\exp\left(f'(u_+)(\xi-\ell)/\e\right)-1}.
\end{equation*}
On the other side, from \eqref{eq:Upm-implicit-unbound} one has
\begin{equation*}
	\frac{\xi-\ell}\e= \int_0^{u_+} \frac{\sqrt{1+(\k_+-f(s))^2}}{\k_+-f(s)} \, ds\geq
	\int_{u_+}^0 \frac{1+\k_+-f(s)}{f(s)-\k_+} \, ds\geq u_++1\int_{u_+}^0 \frac{ds}{\frac{f(u_\pm)}{u_+} s-\k_+},
\end{equation*}
where we used the upper bound in the first estimate of \eqref{boundF}.
Hence we have
\begin{equation*}
	\k_+-f(u_\pm)\leq\frac{f(u_\pm)}{\exp\left(f(u_\pm)(\xi-\ell-u_+)/\e\,u_+\right)-1}.
\end{equation*}
As concerning $\k_-$, again for \eqref{eq:Upm-implicit-unbound} we have
$$
	\frac{\xi+\ell}\e=\int_0^{u_-} \frac{\sqrt{1+(f(s)-\k_-)^2}}{\k_--f(s)} \, ds \leq 
	\int_0^{u_-} \frac{1+\k_--f(s)}{\k_--f(s)} \, ds \leq u_-+1\int_{0}^{u_-} \frac{ds}{\k_- -\frac{f(u_\pm)}{u_-} s}, 
$$
and 
\begin{equation*}
	\frac{\xi+\ell}\e=\int_0^{u_-} \frac{\sqrt{1+(f(s)-\k_-)^2}}{\k_--f(s)} \, ds \geq\int_0^{u_-} \frac{ds}{\k_--f(s)}\geq \int_{0}^{u_-} \frac{ds}{\k_--f(u_\pm)+f'(u_-)(u_--s)}.
\end{equation*}
In conclusion, collecting all the computations, we obtain
\begin{equation}\label{eq:kpm-unbound}
	\begin{aligned}
	\frac{u_+f'(u_+)}{\exp\left(f'(u_+)(\xi-\ell)/\e\right)-1}\leq \k_+ - f(u_\pm) &\leq\frac{f(u_\pm)}{\exp\left(f(u_\pm)(\xi-\ell-u_+)/\e\,u_+\right)-1}, \\
	\frac{u_-f'(u_-)}{\exp\left(f'(u_-)(\xi+\ell)/\e\right)-1}\leq \k_- - f(u_\pm) & \leq \frac{f(u_\pm)}{\exp\left({f(u_\pm)(\xi+\ell-u_-)/{\e \, u_-}}\right)-1}.
	\end{aligned}
\end{equation}
Therefore, $\Omega^\e=\k_--\k_+$ satisfies \eqref{eq:Omega-nonmono} and it is exponentially small for $\e\to0$, uniformly in any compact subset of $J\subset I$.

Furthermore, denoting by $g(\xi):= \k_-(\xi)-\k_+(\xi)$, we observe that $g$ is a monotone decreasing function satisfying
\begin{equation*}
	\lim_{\xi\to-\ell+\e u_-}g(\xi)=+\infty, \qquad \mbox{ and } \qquad \lim_{\xi\to\ell+\e u_+}g(\xi)=-\infty.
\end{equation*}
It follows that there exists a unique value $\bar\xi$ such that $g(\bar \xi)=0$ and 
$U^\e(\cdot;\bar \xi)$ is the unique steady state of the boundary value problem \eqref{eq:unbound-eps}-\eqref{eq:boundary}. 
\end{proof}
As in Remark \ref{rem4.2}, we can prove that, in the case of a Burgers flux $f(u)=u^2/2$, the function $\Omega^\e$ satisfies \eqref{eq:Omega-nonmono-burg}.

\subsection{Linearization}
Let us consider the operator 
\begin{equation*}
	\mathcal{P}^\e[u]:=Q(\e u_x)_x-f(u)_x,
\end{equation*}
defined for $u\in H^2(I)$, and let us fix $U^\e\in C^2(I)$.  
The linearized operator of $\mathcal{P}^\e[u]$ around $U^\e$ (obtained by looking for a solution in the form $u=U^\e+v$, 
being $U^\e$ the function constructed in Propositions \ref{prop:metastablestates-nonmono} and  \ref{prop:metastablestates-unbounded}) is 
\begin{equation}\label{eq:linearized}
	\mathcal{L}^\e v:= \e Q'(\e U^\e_x)v_{xx}+\e^2 Q''(\e U^\e_x)U^\e_{xx}v_x-(f'(U^\e)v)_x.
\end{equation}
\begin{remark}\label{approssimataU}
{\rm
We observe that the functions $U^\e$ defined in Propositions \ref{prop:metastablestates-nonmono} and  \ref{prop:metastablestates-unbounded} 
are actually $H^1$-functions with a continuous derivative up to the jump located at $x=\xi$. 
However, since $C^k$ is dense in $H^1$ for arbitrarily large $k$, we can approximate $U^\e$ with a smooth function up to an arbitrarily small error;
henceforth, from now on we will actually work with a smooth approximation of $U^\e$, still denoted by $U^\e$ for simplicity.
}
\end{remark}

The goal of this subsection is to study the spectral properties of the eigenvalue problem 
\begin{equation}\label{eq:EVP}
	\mathcal{L}^\e v=\lambda v, \qquad \qquad v(\pm\ell)=0.
\end{equation}
To this aim, we rewrite the operator  \eqref{eq:linearized} in the form
\begin{equation*}
	\mathcal{L}^\e v:= pv_{xx}+qv_x+rv,  
\end{equation*}
where
\begin{equation*}
	 p=p(x):=\e Q'(\e U^\e_x(x)), \qquad q=q(x):=p'(x)-f'(U^\e(x)), \qquad r=r(x)=-f''(U^\e)U^\e_x(x).
\end{equation*}
Therefore, we can rewrite the eigenvalue problem \eqref{eq:EVP} in the \emph{Sturm-Liuoville form}
\begin{equation}\label{eq:EVP-ST-L}
	(\rho p v_x)_x+\rho r v=\lambda\rho v, \qquad \qquad v(\pm\ell)=0,
\end{equation}
where the weight function $\rho=\rho(x)$ satisfies
\begin{equation*}
	(\rho p)'=\rho q.
\end{equation*}
By solving the last equation for $\rho$ we find
\begin{equation*}
	\rho(x)=\rho_0\exp\left(-\int_{x_0}^x\frac{f'(U^\e(s))}{p(s)}\,ds\right)=\rho_0\exp\left(-\frac1\e\int_{x_0}^x\frac{f'(U^\e(s))}{Q'(\e U^\e_x(s))}\,ds\right).
\end{equation*}
As done in the previous subsections, we restrict our analysis to the following choices for the dissipation flux function $Q$:
\begin{align}
	&Q(s)=\frac{s}{1+s^2},   &Q'(s)&=\frac{1-s^2}{(1+s^2)^2}, \label{eq:Qnonmono}\\
	&Q(s)=\frac{s}{\sqrt{1-s^2}},  &Q'(s)&=\frac{1}{(1-s^2)^{3/2}}. \label{eq:Qunbound}
\end{align}
We stress once again that the case of a mean curvature type dissipation \eqref{MC} has been already studied in \cite{FGS}.
\begin{proposition}\label{prop:eigenvalues}

The eigenvalue problem \eqref{eq:EVP}, with $\mathcal{L}^\e$ defined in \eqref{eq:linearized} and $Q$ given by \eqref{eq:Qnonmono} or \eqref{eq:Qunbound}
has an infinite sequence of real eigenvalues $\{\lambda^\e_k\}_{k \in \N}$, such that
\begin{equation*}
	\dots < \lambda^\e_3 < \lambda^\e_2 < \lambda^\e_1<0, \quad \mbox{ and } \quad \lambda^\e_k \to -\infty \ \mbox{ as } \ k \to \infty.
\end{equation*}
Moreover, to each eigenvalue $\lambda^\e_k$ corresponds a single eigenfunction $\varphi^\e_k$, having exactly $k-1$ zeros in $I$, 
and the sequence of the eigenfunctions $\{\varphi^\e_k\}_{k \in \N}$ forms an orthogonal basis in the weighted space $L^2_\rho(I)$.
\end{proposition}
\begin{proof}
The proof relies on the fact that we can rewrite \eqref{eq:EVP} in the form \eqref{eq:EVP-ST-L} with $p',r,\rho\in C(I)$ and $p,\rho>0$ in $I$.
Indeed, $p=\e Q'(\e U^\e_x)$ and $Q'>0$ in $(-1,1)$ for both the choices of $Q$ we are considering;
hence, $p>0$  since $-\e^{-1}<U^\e_x<0$.

We can thus apply the classical Sturm--Liouville theory (see \cite[Theorem 2.29]{SL}), providing the existence of an infinite sequence of real eigenvalues converging at $-\infty$, 
with corresponding eigenfunctions satisfying the properties of the statement.

In order to show that all the eigenvalues are negative, we study the sign of the first eigenvalue;
we assume, without loss of generality, that $\varphi^\e_1 >0$ in $I$ and we integrate the relation $\mathcal L^\e\varphi^\e_1= \lambda^\e_1 \varphi^\e_1$.
Because of the expressions of $p,q,r$, it follows
\begin{equation}\label{segnolambda1}
	\begin{aligned}
	\lambda^\e_1 \int_I \varphi^\e_1 \, dx &=
	p\,{\varphi^\e_1}'\big|^{\ell}_{-\ell}-(f'(U^\e)\varphi^\e_1)\big|^{\ell}_{-\ell}\\
	&=p(\ell){\varphi^\e_1}'(\ell)-p(-\ell){\varphi^\e_1}'(-\ell),
	\end{aligned}
\end{equation}
where in the last line we used  $\varphi^\e_1(\pm \ell)=0$. 
Since $\int_I \varphi^\e_1 >0$, $p>0$ and ${\varphi^\e_1}'(\pm\ell)\lessgtr0$, it follows $\lambda^\e_1 <0$, and the proof is complete.
\end{proof}

\subsubsection{Asymptotics for the first eigenvalue}
\smallbreak
We want to give an estimate for the first eigenvalue $\lambda_1^\e$ of the linearized operator \eqref{eq:linearized},
so that to show that the spectral properties stated in {\bf Step III} of the strategy are indeed satisfied. 
For simplicity, we here show the computations only in the case of the nonmomotone dissipation
$$Q(s)=\frac{s}{1+s^2},$$
being the case of an unbounded dissipation completely analogous. 
Also, for the sake of simplicity, we consider a flux function of Burgers type, i.e., $f(u)={u^2}/2$; 
we stress that in this case the boundary data $u_\pm= \mp u_*$ for some $u_*>0$ (see condition \eqref{eq:f-metastability}) 
and that condition \eqref{eq:condnonmono-eps} gives the constraint $u_* < 1$.

By integrating the relation $\mathcal L^\e\varphi^\e_1= \lambda^\e_1 \varphi^\e_1$, we obtain, as in \eqref{segnolambda1}
\begin{equation}\label{roughlambda1}
\lambda_1^\e\int_I\varphi^\e_1\,dx  =p\,{ \varphi^\e_1}' \big|^{\ell}_{-\ell}.
\end{equation}
Moreover, since $U^\e_{\pm}$ are stationary solutions in the intervals $(-\ell,\xi)$ and $(\xi, \ell)$ respectively, 
we deduce
\begin{equation*}
	\e \d_x U^\e_{\pm}= \frac{2 \left(f(U^\e_{\pm})-\k_{\pm}\right)}{{1+ \sqrt{1-4(f(U^\e_{\pm})-\k_{\pm})^2}}}
\end{equation*}
which, by using  \eqref{eq:kpm-nonmono}, leads to
\begin{equation*}
	\d_x U^\e(\pm \ell) = \frac{2\e^{-1} \left(f(u_{\pm})-\k_{\pm}\right)}{{1+ \sqrt{1-4(f(u_{\pm})-\k_{\pm})^2}}}
 \approx   -\frac{2 u_*^2}{2 \e} \, e^{-u_*(\ell-|\xi|)/{\e}}.
\end{equation*}
Going further, we can state that $\varphi_1^\e \approx U_x^\e$ (for a proof of this statement, we refer the reader to \cite[Lemma 5.3]{FGS} and \cite[Remark 5.4]{FGS}); hence
$${\varphi^\e_1}'(\pm\ell) \approx U^\e_{xx} (\pm \ell),$$
where
\begin{equation*}
	U^\e_{xx}(\pm \ell) = \frac{(1+{\e^2U^\e_x{(\pm \ell)}^2)}^2}{1-{\e^2U_x^\e(\pm \ell)}^2} \, U^\e(\pm \ell) U^\e_x(\pm \ell)\approx 
	\frac{u_*^3}{\e}\, e^{-u_*(\ell-|\xi|)/\e},
\end{equation*}
and 
\begin{equation*}
	\int_I\varphi_1^\e\,dx\approx U^\e(x)\Big|^{\ell}_{-\ell}=2u_*.
\end{equation*}
Recalling that $p(x) \approx \e$, from \eqref{roughlambda1} we finally  have
\begin{equation}\label{rough}
	\lambda_1^\e(\xi)\approx-u_*^2 e^{-u_*(\ell-|\xi|)/\e}.
\end{equation}
We stress that, since the large time behavior of the solution is heuristically dictated by terms of the order $e^{\lambda_1^\e t}$, 
we expect $\lambda_1^\e$ to give a good approximation of the speed rate of convergence of the solution towards its asymptotic configuration (as in the linear case, see \cite{MS}). 
In particular, we expect to have a metastable behavior whenever $\lambda_1^\e$ is exponentially small in $\e$.

\subsection{Conclusions}\label{sec:4concl}
Having proved that the strategy developed in \cite{MS} is applicable (i.e. having proved {\bf Steps I-II-III} and the hypotheses therein), we can proceed as in \cite{FGS}; 
in particular, as done in \cite{FGS}, we can apply  \cite[Theorem 3.4]{Str17} to prove that the following estimate for the perturbation $v$
\begin{equation}\label{stimav}
\|v\|_{{}_{L^2}}(t) \leq  \|v_0\|_{{}_{L^2}} e^{\nu^\e t}+ c \, t \,  |\Omega^\e|_{{}_{{}_{L^\infty}}}, \quad \nu^\e := c|\Omega^\e|_{{}_{L^\infty}}-\sup_{\xi} |\lambda_1^\e(\xi)|,
\end{equation}
holds. 
Such estimates state that the $L^2$--norm of $v$ is exponentially small as $\e \to 0$ and for large time (provided $\|v_0\|_{{}_{L^2}}$ to be small enough), 
up to a reminder that is measured by $\Omega^\e$, which is again exponentially small in $\e$ 
(see \eqref{eq:Omega-nonmono}).

In particular, estimate \eqref{stimav} can be used to decouple the  following ODE for the variable $\xi=\xi(t)$ 
(for the details on how it can be obtained, see again \cite{FGS, Str17})
\begin{equation*}
	\frac{d\xi}{dt} = \langle \psi^\varepsilon_1,{\mathcal P^\e[U^{\varepsilon}] \rangle}(1+v) \leq  \langle \psi^\varepsilon_1,{\mathcal P^\e[U^{\varepsilon}] \rangle}(1+r), 
\end{equation*}
with $|r|\leq C \left( \|v_0\|_{{}_{L^2}}e^{\nu^\e t} + t|\Omega^\e|_{{}_{L^\infty}}\right)$, and where $\psi_1^\e$ is the first eigenfunction of the adjoint operator  $\mathcal L^{\e,*}$. 
Since $r$ is exponentially small as $\e\to0$, the speed rate of $\xi(t)$ is thus asymptotically given by  
\begin{equation}\label{eq:xi}
	\frac{d\xi}{dt} \approx  \langle \psi^\varepsilon_1,{\mathcal P^\e[U^{\varepsilon}] \rangle}= \psi^\varepsilon_1(\xi)\left(\k_-(\xi)-\k_+(\xi)\right)=:\theta^\e(\xi).
\end{equation}
We can prove that $\psi_1^\e(\xi)\approx C/(u_--u_+)$ as $\e \to 0^+$ for any $\xi\in(-\ell,\ell)$, 
so that the dynamics of $\xi$ as $\e \to 0$ is governed by the difference $g(\xi):=\k_-(\xi)-\k_+(\xi)$. 
We have already studied the properties of the function $g$;
it is a monotone decreasing function such that there exists a unique $\bar\xi\in(-\ell,\ell)$ with $g(\bar\xi)=0$. 
Moreover, it is exponentially small in $\e$ (see Propositions \ref{prop:metastablestates-nonmono} and \ref{prop:metastablestates-unbounded}).

Hence, if looking at $\xi(t)$ as the position of the (unique) interface of the solution $u$ to the original systems 
(respectively \eqref{eq:nonmono-eps}-\eqref{eq:boundary} and \eqref{eq:unbound-eps}-\eqref{eq:boundary}), we have convergence of $\xi$ to $\bar\xi$; 
the speed rate of such convergence is given by \eqref{eq:xi}, and so by the magnitude of $\theta^\e(\xi)$, which is exponentially small in $\e$. 
Because of the decomposition $u(x,t)=U^\e(x;\xi(t))+v(x,t)$, it follows that  $u$ is converging to $U^\e(\cdot;\bar\xi)$ for large times, 
being $U^\e(\cdot;\bar\xi)$ the unique steady state of the system, with exponentially slow speed, leading to a metastable behavior.
\vskip0.2cm

For example, in the case of the boundary value problem \eqref{eq:nonmono-eps}-\eqref{eq:boundary} with $f(u)=u^2/2$, we have
\begin{equation}\label{stocazzo}
	|\theta^\e(\xi)|\leq C  \, u_*\exp(-u_*(\ell-|\xi|)/\e), \qquad \qquad \xi \in I,
\end{equation}
where we used  \eqref{eq:Omega-nonmono-burg} and $\psi_1^\e(\xi)\approx C/u_*$. 
Hence, if choosing the boundary data so that assumption \eqref{eq:condnonmono-eps} is satisfied, 
we have that the speed rate of convergence of the solution is {$\mathcal O(e^{-u_*/\e})$}. 
It is worth notice, however, that if the boundary data are taken too small (for example $u_*=\e$), 
then estimate \eqref{stocazzo} only gives a speed rate of convergence of the order $\e{}$, and no metastability will be observed. 
This is consistent with the stability properties of the steady state proved in Theorem \ref{thm:stab-nonmono}, and in particular with assumption \eqref{eq:ipo-stab-nonmono}: 
if $u_*$ is too small, then \eqref{eq:ipo-stab-nonmono} is satisfied and we enter the setting of Theorem \ref{thm:stab-nonmono}, 
that is we have fast convergence towards the equilibrium.
The same discussion can be done in the case \eqref{eq:unbound-eps}-\eqref{eq:boundary},
where the boundary data are not subject to any smallness condition.
Let us remark that also in this case if we choose $u_*=\e$, no metastability will be observed.
In particular, if $u_*=\e$ then in both the cases \eqref{eq:nonmono-eps} and \eqref{eq:unbound-eps} we obtain the following
estimate for $g$:
\begin{equation*}
	|\theta^\e(\xi)|\leq C  \, \e \gamma(\xi),
\end{equation*}
where $\gamma$ does not depend on $\e$.
It follows that, if we consider the new variable $\tau=\e^{-1}t$ and the function $\tilde u(x,\tau)=u(x,\e\tau)$, 
where $u$ is a solution to \eqref{NonLBurg}, then $\tilde u$ solves
\begin{equation}\label{eq:rescale}
	\tilde u_\tau=\e^{-1}{Q(\e\tilde u_x)}_x-\e^{-1} {f(\tilde u)}_x,
\end{equation}
and its speed rate of convergence does not depend on $\e$. 
We will show some numerical evidence of this fact in the next section.

\section{Numerical solutions}\label{sec5}
In this section, we illustrate some numerical simulations for the time-dependent solution  to the following  problems
\begin{align}
	u_t & =\left(\frac{\e u_x}{{1+ \e^2u_x^2}}\right)_x - f(u)_x, &\quad &x\in I, t > 0, \label{eq:QnonmonoNUM}\\
	u_t & =\left(\frac{\e u_x}{{\sqrt{1-\e^2 u_x^2}}}\right)_x - f(u)_x, &\quad &x\in I, t > 0,  \label{eq:QunboundNUM}
\end{align}
complemented with boundary and initial conditions 
\begin{equation}\label{BIC}
		u(\pm\ell,t)=u_\pm \ \ \mbox{for } \ t >0, \qquad \mbox{and} \qquad u(x,0) =u_0(x)  \ \ \mbox{for } \  x\in I.
\end{equation}
The goal is to numerically compute the speed rate of convergence of the solutions and to show that a metastable behavior indeed appears, 
so that to give evidence of the rigorous results of Section \ref{sec4}. 

As before, the flux function $f$ and the boundary data $u_\pm$ are required to satisfy \eqref{eq:f-metastability} 
(in the case \eqref{eq:QnonmonoNUM} we choose the boundary data so that also condition \eqref{eq:condnonmono-eps} is satisfied);
notice that, as in the linear case, the assumptions in \eqref{eq:f-metastability} are needed in order to observe a metastable behavior 
(that is, in order for the steady state to be {\it metastable}), while they are not necessary for its existence
(see \cite[Section 4]{FGS} for some numerics in the case of a dissipation flux function \eqref{MC} when such assumptions are violated).  
To fix the ideas, in all the numerical examples we choose the Burgers flux $f(u)=u^2/2$ and the boundary data $u_\pm=\mp u_*$ for some $u_*>0$. 
In such a way, condition \eqref{eq:f-metastability} is satisfied and we will show that the appearance of the metastable dynamics depends on the choice of $u_*$.
With this choice of $f$, condition \eqref{eq:condnonmono-eps}, which is needed to have existence of the steady states in the case \eqref{eq:QnonmonoNUM}, 
becomes $u_*<1$, while  conditions \eqref{eq:ipo-stab-nonmono}, \eqref{eq:ipo-stab-unbound}, 
which ensure stability of the steady states in the cases \eqref{eq:QnonmonoNUM} and \eqref{eq:QunboundNUM}, read as
\begin{equation*}
	u_*\leq \Gamma_1\e \qquad\mbox{and} \qquad u_*\leq\frac{\pi^2\e}{4\ell^2}.
\end{equation*}
We stress once again that, if the boundary data are taken of the order $\e$ (hence too small) 
we enter the setting of Theorems \ref{thm:stab-nonmono}-\ref{thm:stab-unbound}, and the steady state is stable but not metastable.
\subsection{The non monotone case}

Figure \ref{fig:eps005nonmono} shows the metastable behavior occurring for the solutions to \eqref{eq:QnonmonoNUM}-\eqref{BIC}
and gives a flavor of how the size of the parameter $\e$ influences the speed rate of convergence of the solution towards the steady state; 
we can clearly see that, as soon as $\e$ becomes smaller (right picture), the time needed to reach the equilibrium becomes much bigger 
(compare the same position $x \sim 0$ for the interface reached for times of the order $10^6$ when $\e=0.05$ 
and for times of the order $10^{11}$ when $\e=0.025$).

\begin{figure}[h] 
\centering
\includegraphics[width=7cm,height=5.5cm]{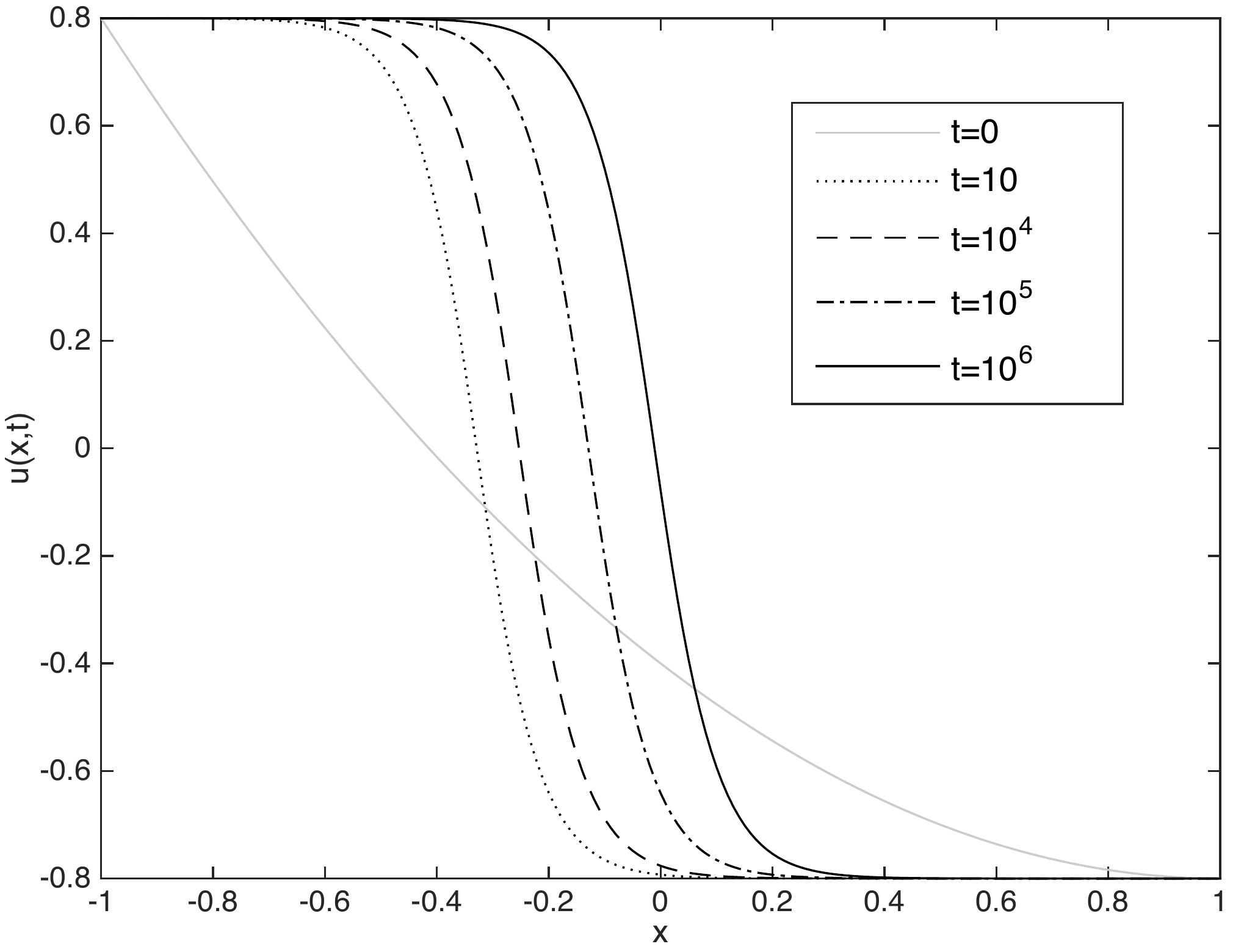}
 \hspace{3mm}
\includegraphics[width=7cm,height=5.5cm]{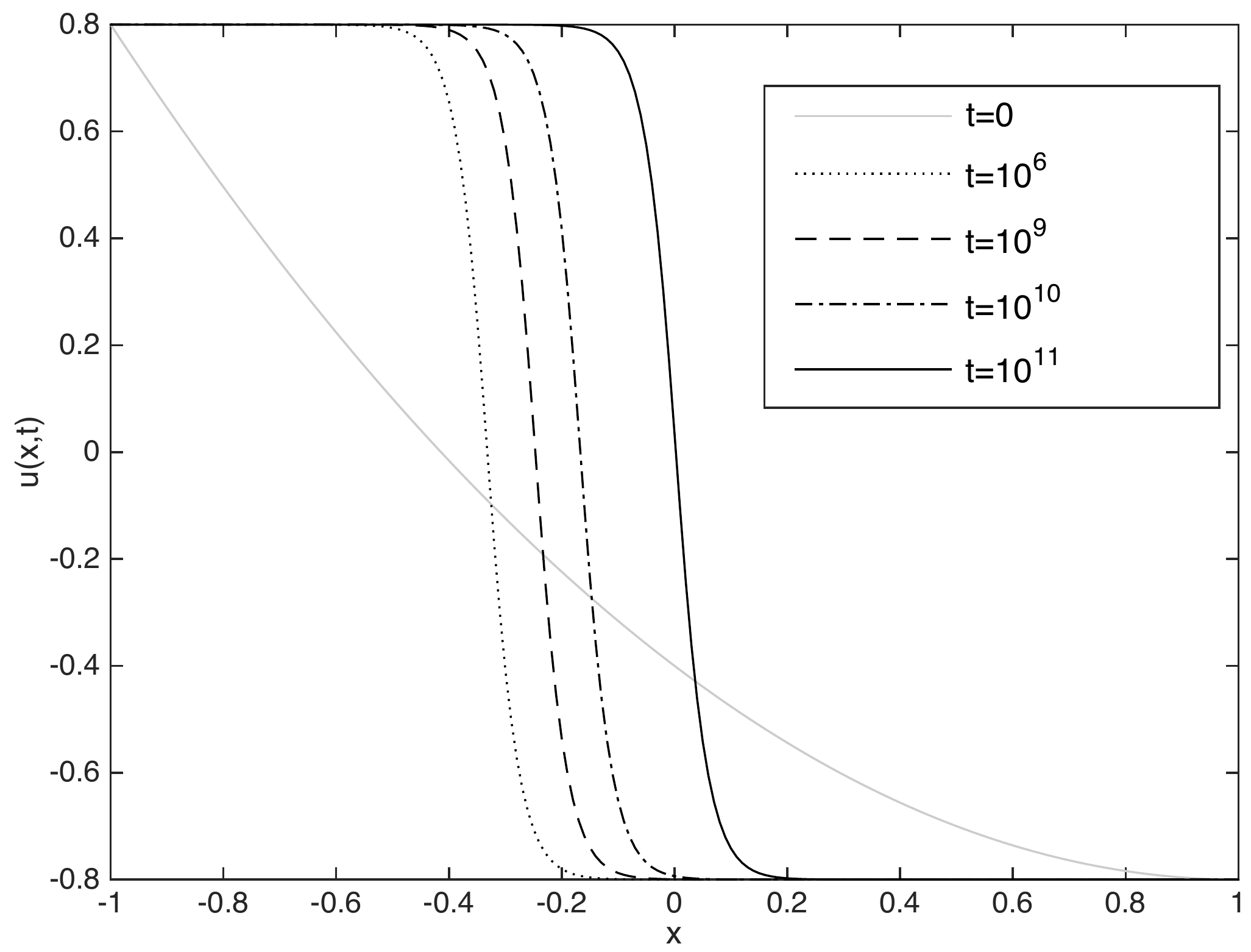}
 \hspace{3mm}
 \caption{\small{The dynamics of the solution to the IBVP \eqref{eq:QnonmonoNUM}-\eqref{BIC} with 
 initial datum $u_0(x)=0.8\left(\frac{1}{2}x^2-x-\frac{1}{2}\right)$ and $u_*=0.8$. 
 In the left picture $\e=0.05$, in the right one $\e=0.025$.}}\label{fig:eps005nonmono}
 \end{figure}
 
In Figure \ref{fig:nonmonodiscon} (left hand side), 
we plot what happens when considering {\it discontinuous initial data} with a {\it positive zero};  
the solution becomes smooth in short times, and we still observe a metastable behavior, 
with the interface moving towards the left (with negative speed) to reach its asymptotic configuration.
In the right hand picture we see what happens if starting from initial data which are nor {\it decreasing } nor such that $|u_0(x)| < u_*$.
We can clearly see that the latter are no necessary conditions for the appearance of a metastable behavior; 
indeed, the solution starting from such initial configuration develops into a decreasing function $u$ such that $|u|< u_*$ in short times, 
and then experiences the same metastable behavior.

 \begin{figure}[ht]
\centering
\includegraphics[width=7cm,height=5.5cm]{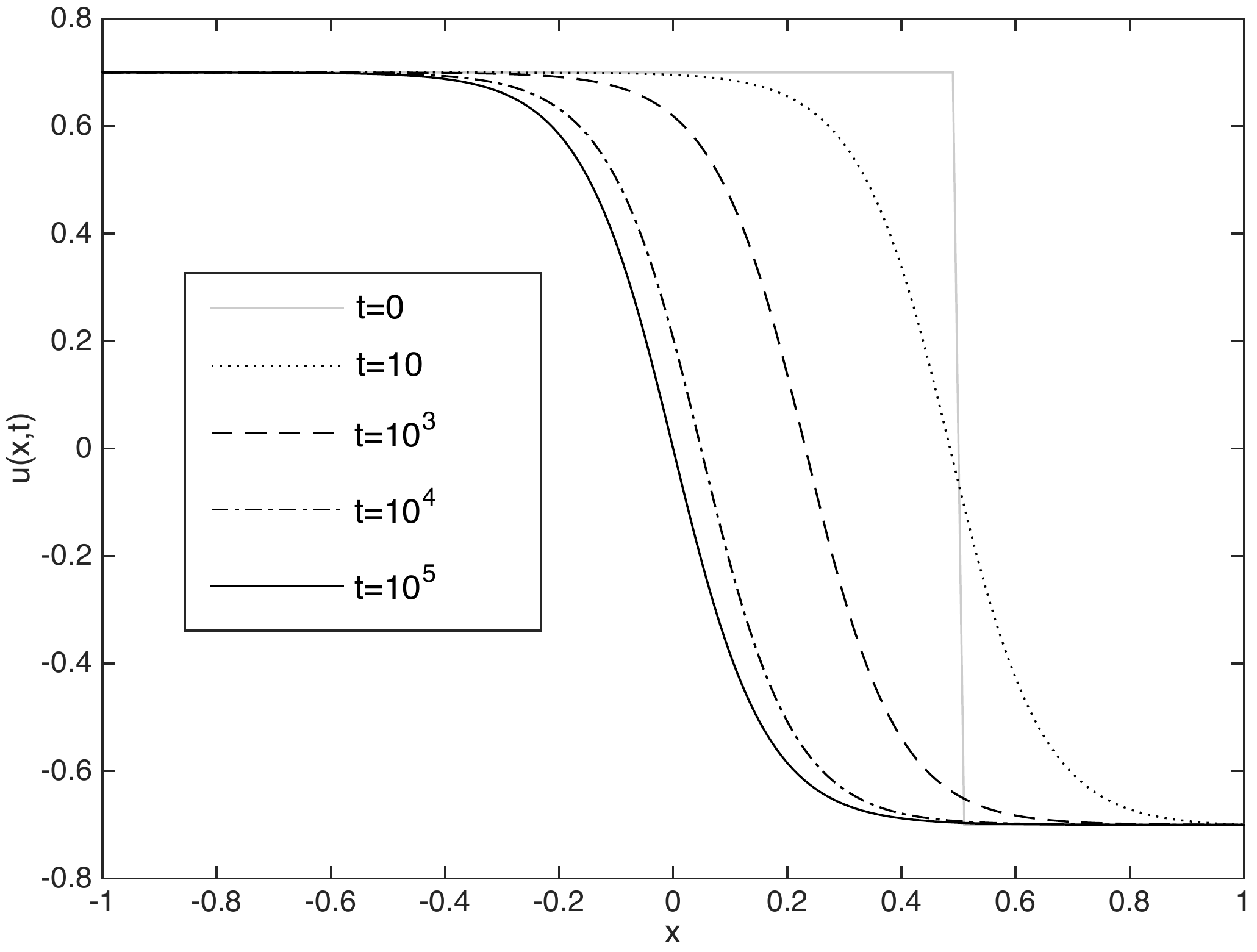}
\hspace{3mm}
\includegraphics[width=7cm,height=5.5cm]{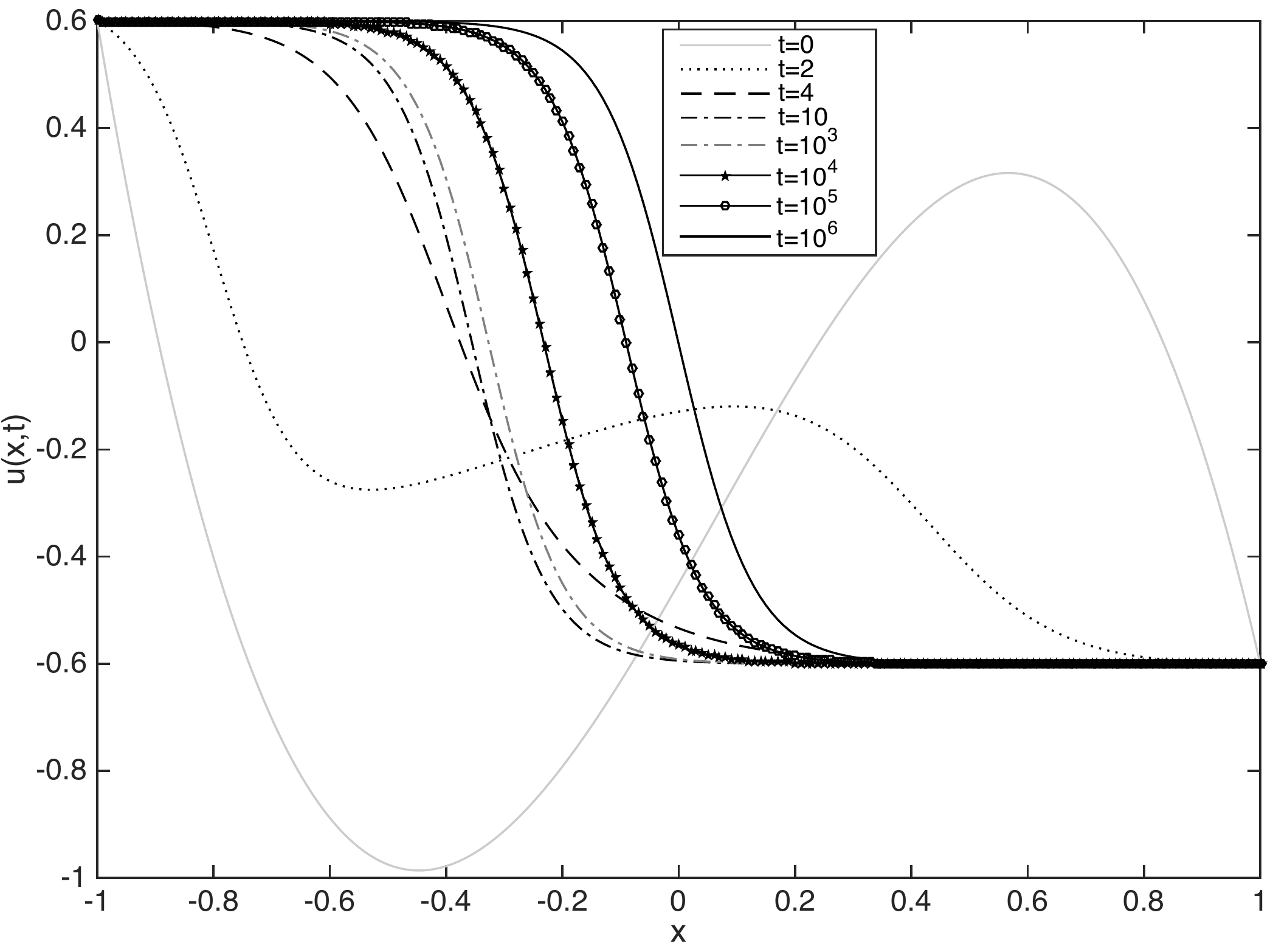}
 \caption{\small{The dynamics of the solution to the IBVP \eqref{eq:QnonmonoNUM}-\eqref{BIC}.
In the left picture we consider a discontinuous initial datum with $\e=0.06$ and $u_*=0.7$, in the right picture the non-monotone one 
$u_0(x)=0.6\left(-\frac{25}{6}x^3+\frac{3}{4}x^2+\frac{19}{6}x-\frac{3}4\right)$ with $\e=0.04$ and $u_*=0.6$.}}\label{fig:nonmonodiscon}
 \end{figure}

Furthermore, in Figure \ref{fig:nonmonodiscon}, the times needed to reach the equilibrium are smaller than the ones of Figure \ref{fig:eps005nonmono},
because we choose different $\e$ and $u_*$. 
Notice that in the left picture $e^{u_*/\e}\approx10^5$, while in the right one $e^{u_*/\e}\approx10^6$.

\subsection{The unbounded case}
We now consider the boundary value problem \eqref{eq:QunboundNUM}-\eqref{BIC}. 
In this case, we have to choose initial conditions $u_0$ such that $|u_0'|<\e^{-1}$; 
in particular, discontinuous initial data are here prohibited, as opposite to the previous case.

As before, in Figure \ref{fig:eps004unbound} we show how the size of the viscosity parameter $\e$ influences the speed rate of convergence; 
to one side, we have convergence for times of the order $10^6$ for $\e=0.04$ (left hand picture), 
while when $\e=0.02$, we have to wait till $t \sim 10^{12}$ for the solution to reach the equilibrium.

\begin{figure}[ht]
\centering
\includegraphics[width=7cm,height=5.5cm]{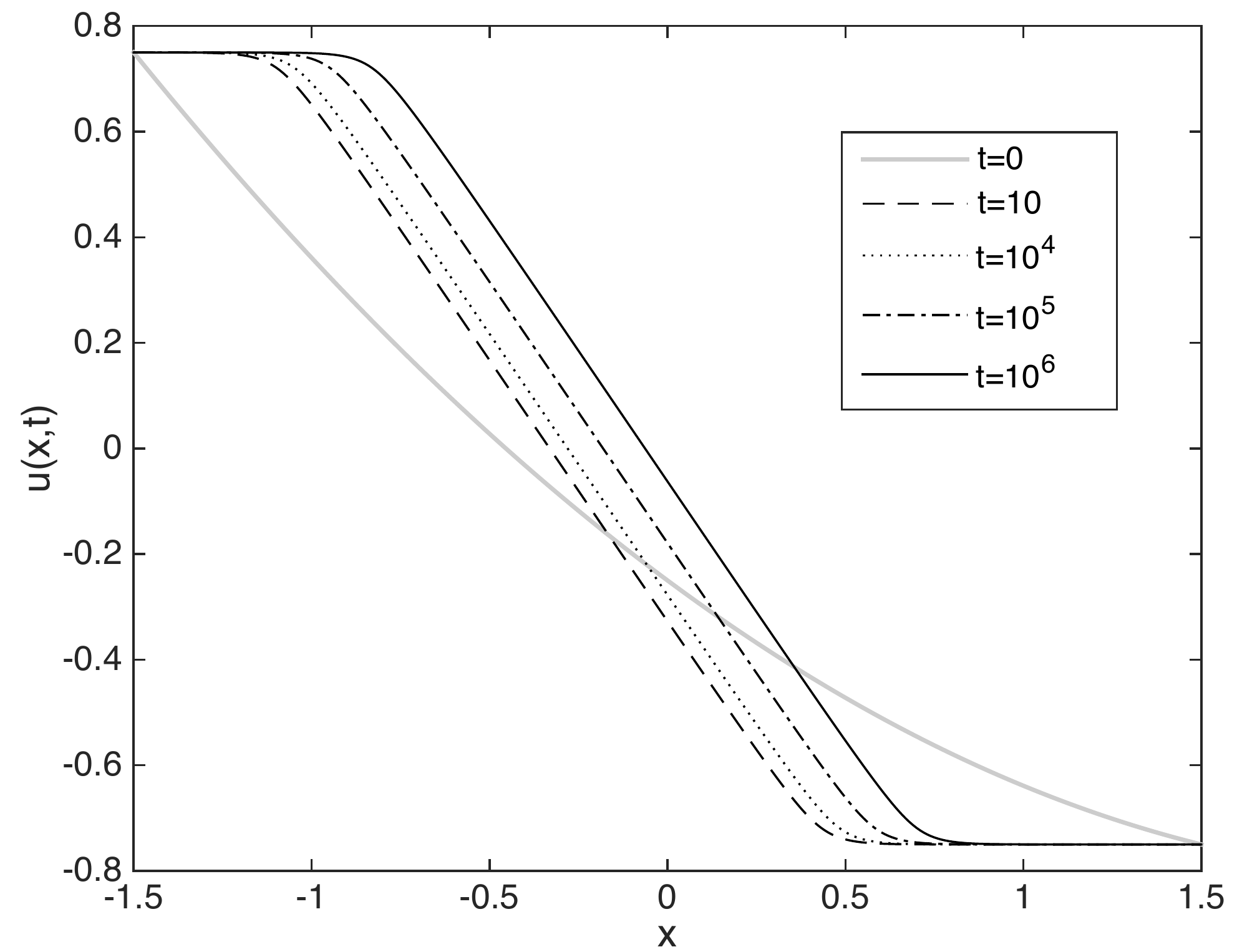}
 \hspace{3mm}
\includegraphics[width=7cm,height=5.5cm]{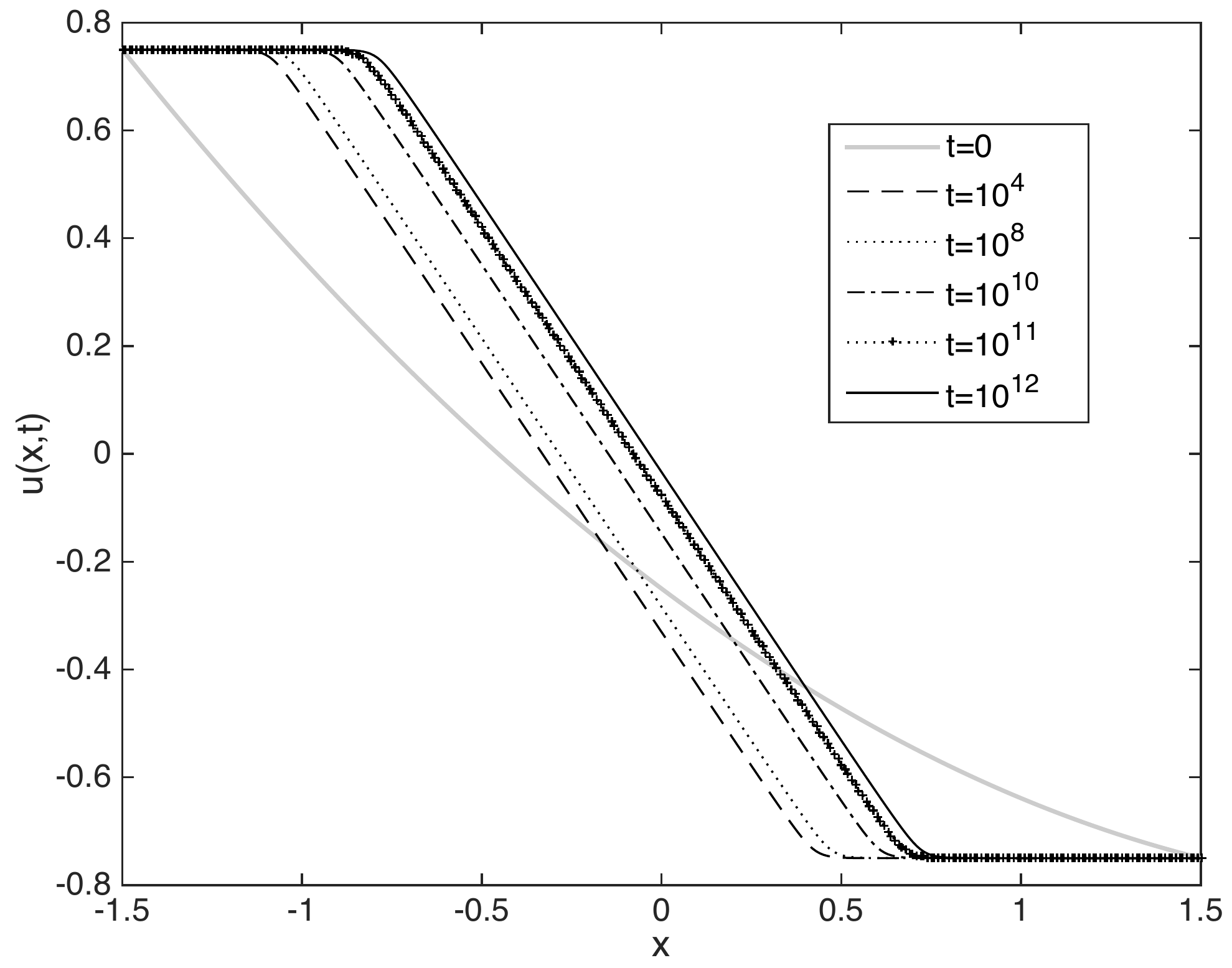}
 \hspace{3mm}
 \caption{\small{The dynamics of the solution to the IBVP \eqref{eq:QunboundNUM}-\eqref{BIC} with $u_* = 0.75$
 and initial datum $u_0(x)=\frac{1}{9}x^2-\frac12x-\frac{1}{4}$. In the left picture $\e=0.04$, in the right one $\e=0.02$.}}\label{fig:eps004unbound}
 \end{figure}

In Figure \ref{fig:eps007unbound}, we illustrate what happens when choosing nonmonotone initial data; 
the solution becomes monotone in short times, and, after an interface is formed, it (slowly) converges towards its asymptotic configuration.
  
\begin{figure}[ht]
\centering
\includegraphics[width=7cm,height=5.5cm]{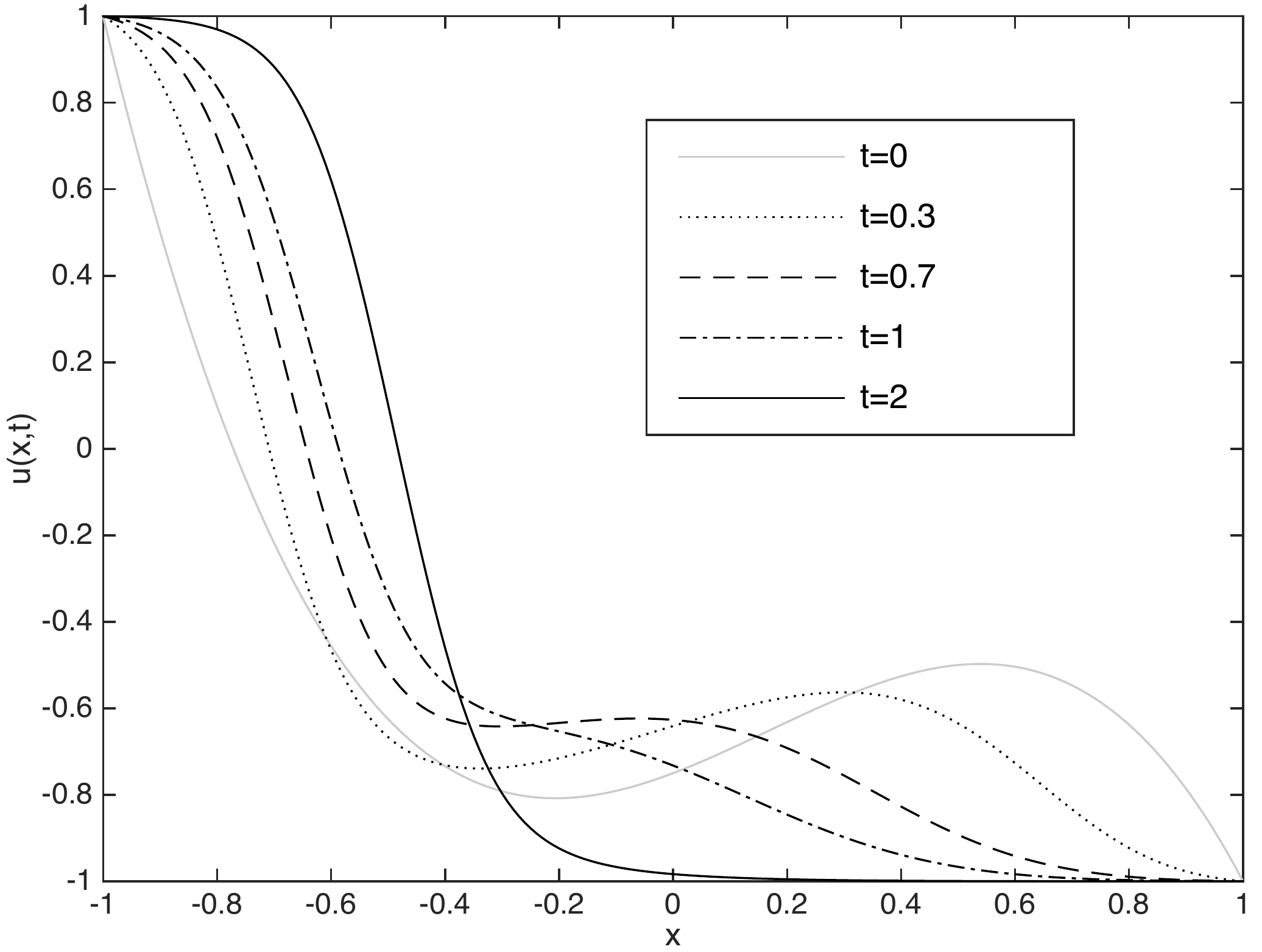}
 \hspace{3mm}
\includegraphics[width=7cm,height=5.5cm]{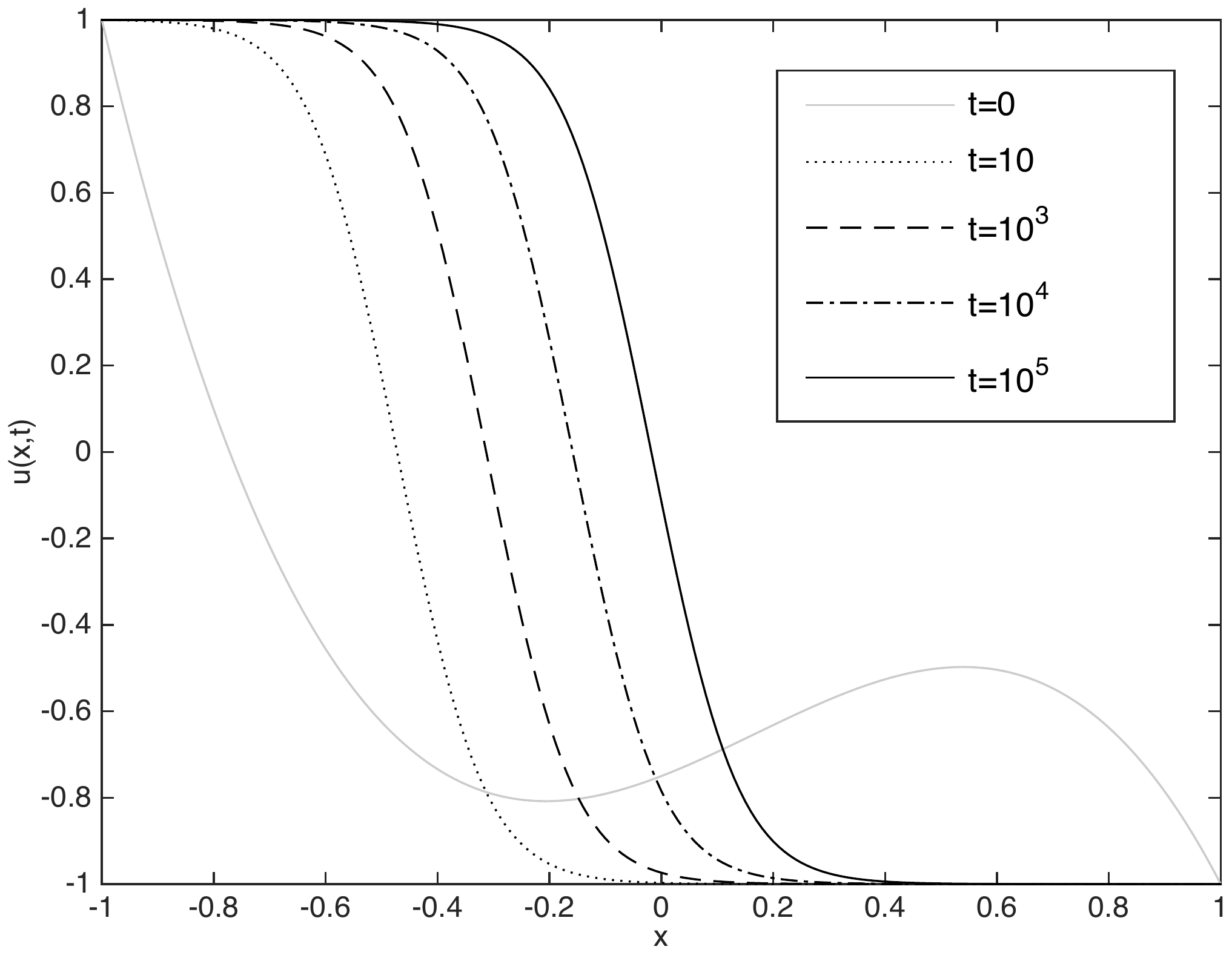}
 \hspace{3mm}
 \caption{\small{The dynamics of the solution to the IBVP \eqref{eq:QunboundNUM}-\eqref{BIC} with $\e=0.07$, $u_*=1$
 and initial datum $u_0(x)=-\frac32x^3+\frac34x^2+\frac12x-\frac34$.}}\label{fig:eps007unbound}
 \end{figure}
 
To conclude, in Figure \ref{fig:datieps} we show what happens if choosing the boundary data $u_*$ too small. 
If considering \eqref{eq:QnonmonoNUM} and choosing $u_*\leq\e$ (hence smaller than Figure \ref{fig:eps005nonmono}) 
we see that no metastable behavior occurs (left hand picture of Figure \ref{fig:datieps}).  
In this case the time taken for the solution to reach its asymptotic limit becomes much smaller, and we thus have a fast
convergence towards the equilibrium rather than a metastable behavior. 

This is also consistent with the estimate obtained for the first eigenvalue in \eqref{rough};
if the boundary data are taken too small (for example of the order $\e$), then $\lambda_1^\e = \mathcal O(\e^2)$, and no metastability is expected.
The same happens for the solution to \eqref{eq:QunboundNUM}; 
here no smallness condition on the boundary data is necessary for the existence of a steady state (see Proposition \ref{prop:ex-unbound}). 
However, if choosing $u_* =\e/2$, we have convergence towards the equilibrium for times of the order $t=10^3$ (right hand picture in Figure \ref{fig:datieps}).

\begin{figure}[ht]
\centering
\includegraphics[width=7cm,height=5.5cm]{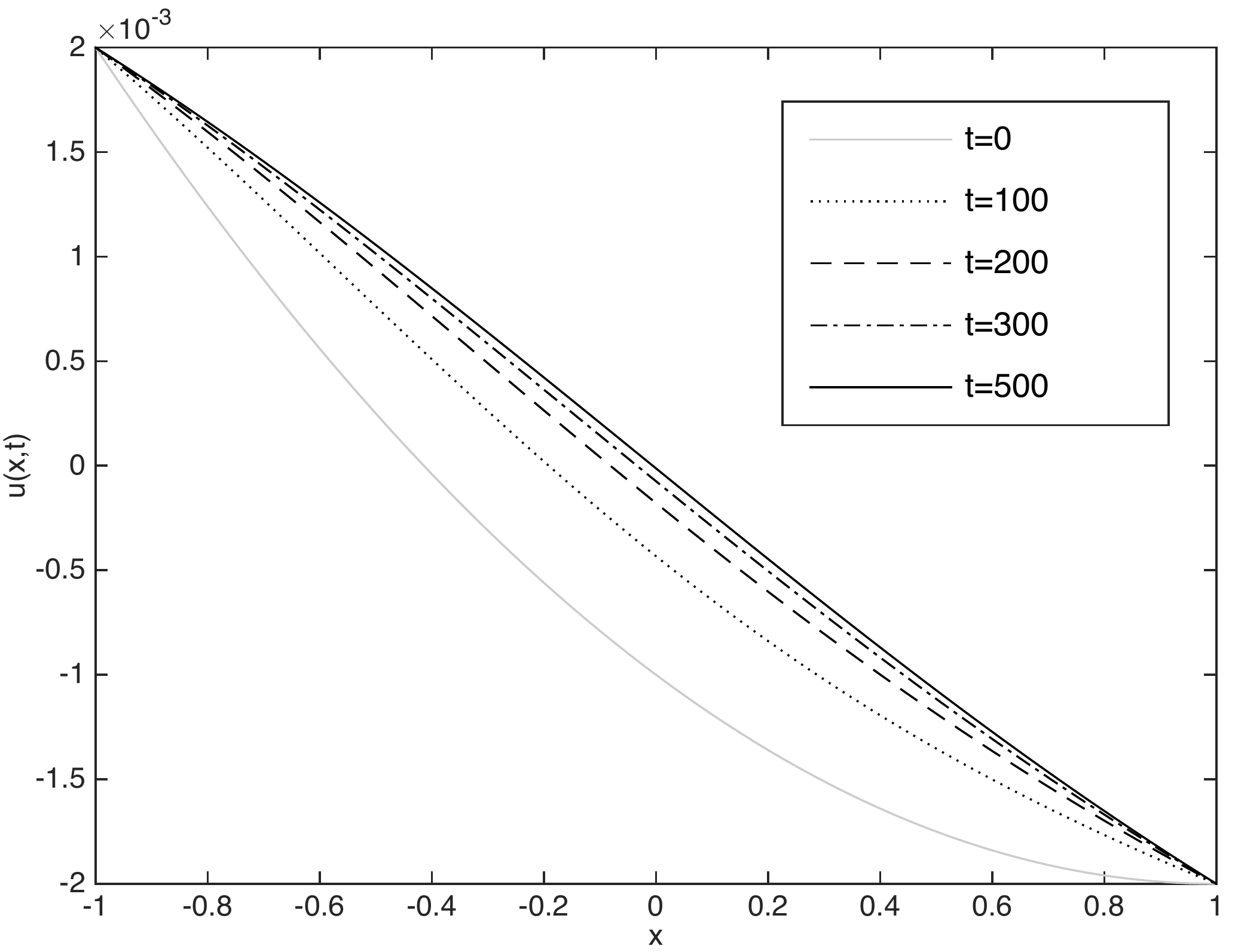}
\hspace{3mm}
\includegraphics[width=7cm,height=5.5cm]{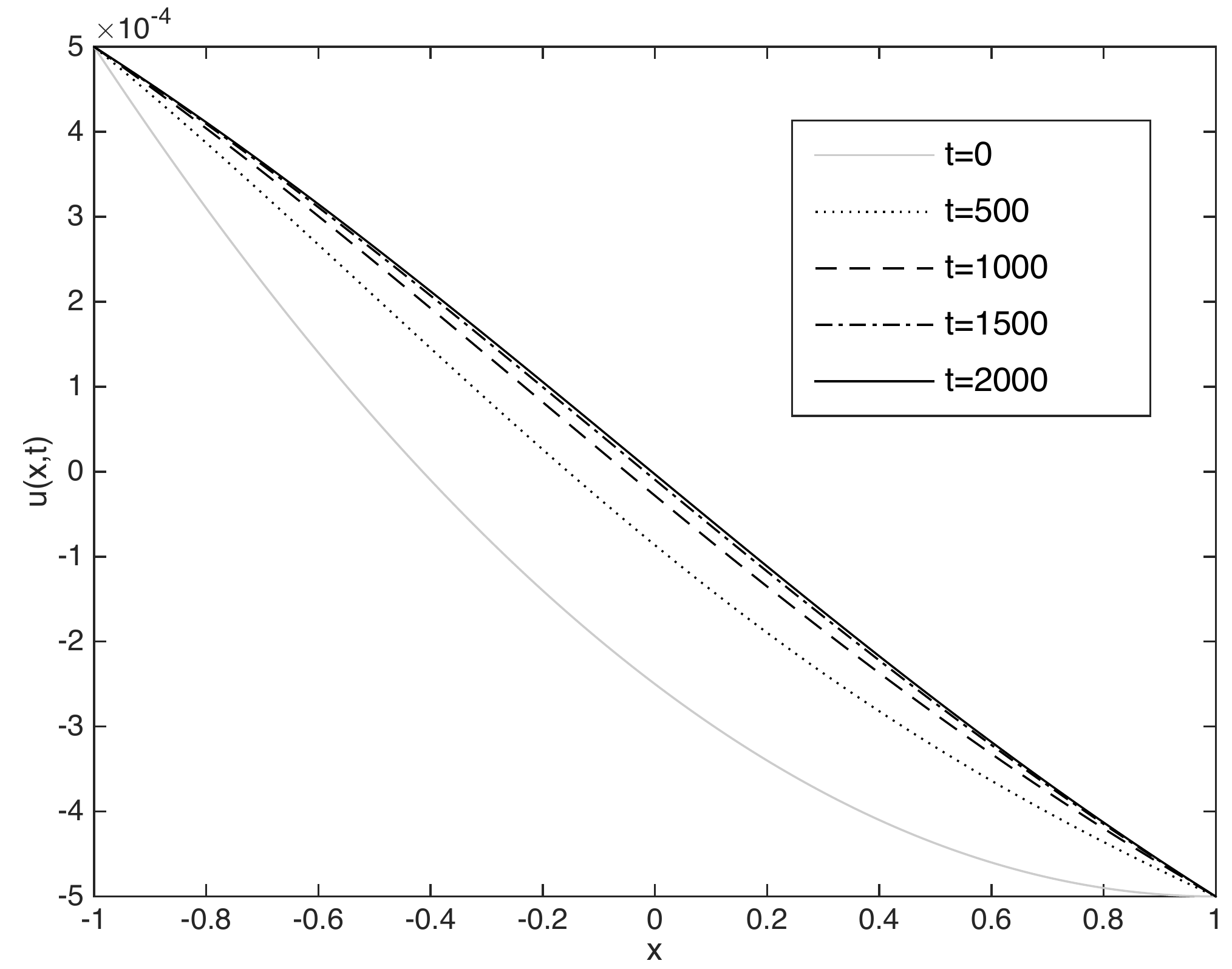}
\hspace{3mm}
\caption{\small{The dynamics of the solution to the IBVP \eqref{eq:QnonmonoNUM}-\eqref{BIC}  with $\e=0.004$ (left),
and \eqref{eq:QunboundNUM}-\eqref{BIC}  with $\e=0.001$ (right);
in both case we choose $u_*=\e/2$ and the initial datum $u_0(x)=\frac\e2\left(\frac12x^2-x-\frac12\right)$.}}
\label{fig:datieps}
\end{figure}

As we showed at the end of Section \ref{sec4}, the velocity of the interface is proportional to $\e$ if the initial data are sufficiently small ($u_*=\e$).
It follows that, if we consider the rescaled version \eqref{eq:rescale} for different values of the parameter $\e$, 
then the time taken for the solutions to reach the asymptotic limit is fixed (independent of $\e$).
This fact is confirmed by the plots in Figure \ref{fig:unboundrescale}, where we show two numerical solutions of \eqref{eq:rescale},
corresponding to two different values of the parameter $\e$, in the case of an unbounded flux $Q$ \eqref{eq:Qunbound}.
 
\begin{figure}[ht]
\centering
\includegraphics[width=7cm,height=5.5cm]{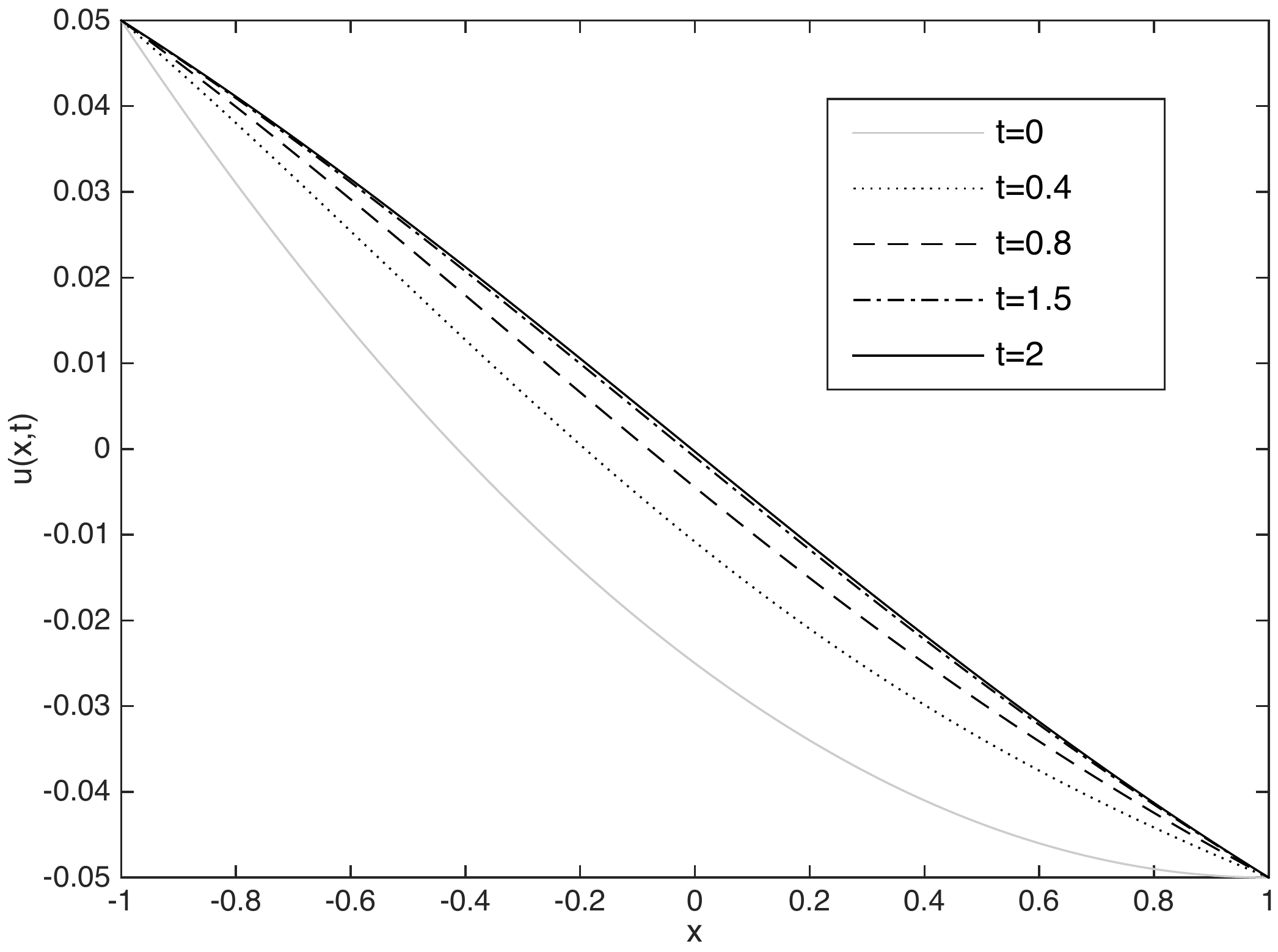}
\hspace{3mm}
\includegraphics[width=7cm,height=5.6cm]{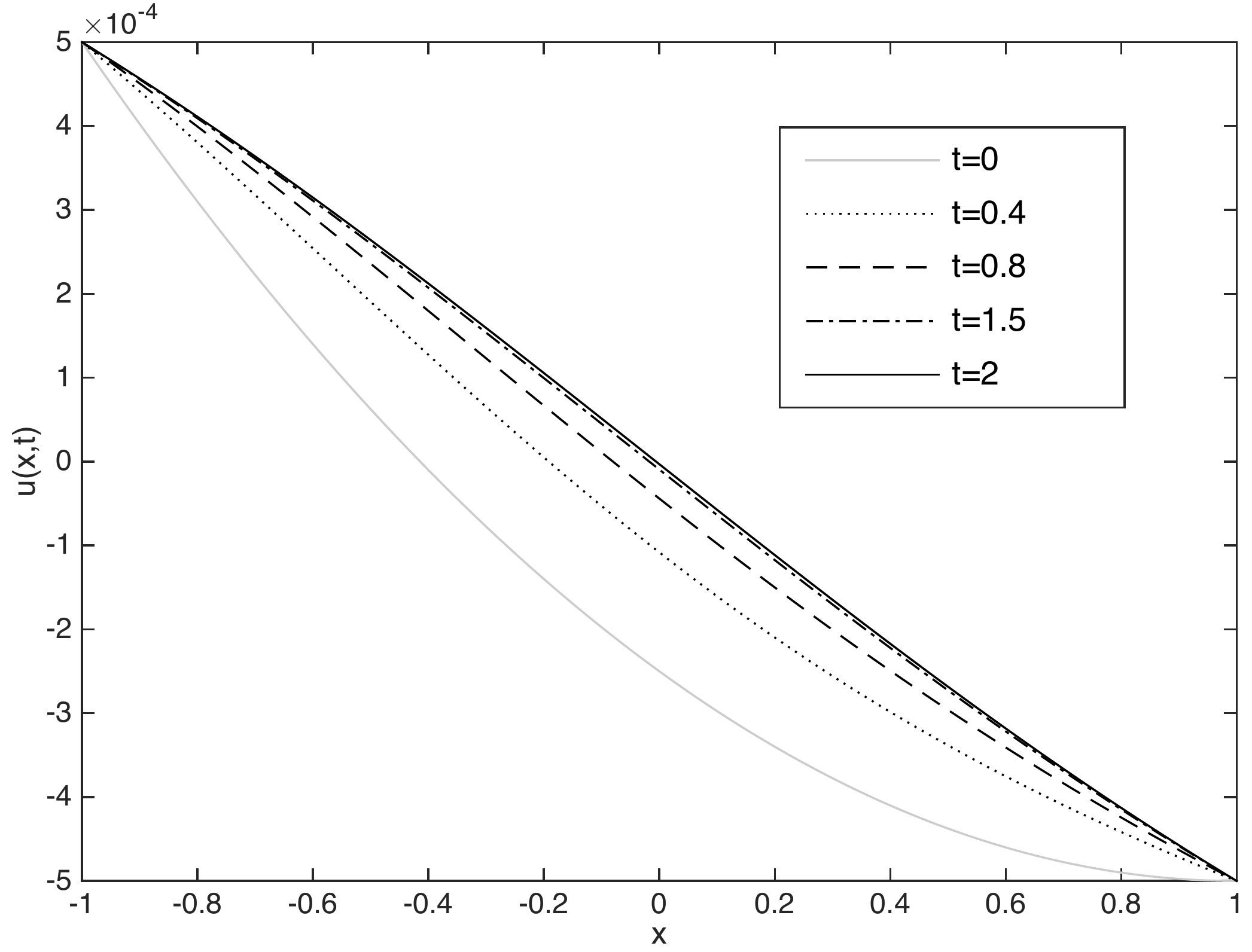}
\hspace{3mm}
\caption{\small{The dynamics of the solution to the IBVP \eqref{eq:rescale}-\eqref{BIC} with $Q(s)=\frac{s}{\sqrt{1-s^2}}$, 
$u_*=\e/2$ and initial datum $u_0(x)=\frac\e2\left(\frac12x^2-x-\frac12\right)$.
In the left hand picture we choose $\e=0.1$; in the right one $\e=0.001$.}}
\label{fig:unboundrescale}
\end{figure}

\end{document}